\documentclass[10pt,a4paper]{amsart}
\usepackage{amssymb}
\usepackage{latexsym}
\usepackage{exscale}
\usepackage{amsfonts}
\usepackage{graphicx}
\usepackage{mathrsfs}
\usepackage{amsmath,amscd,amsthm}
\usepackage{bbm,pifont}
\usepackage{enumerate}
\usepackage{color}
\usepackage{amssymb}
\usepackage{latexsym}
\usepackage{exscale}

\newtheorem{theorem}{Theorem}[section]

\newtheorem{lemma}[theorem]{Lemma}
\newtheorem{definition}[theorem]{Definition}

\newtheorem{remark}[theorem]{Remark}
\numberwithin{equation}{section}
\usepackage{amssymb}
\usepackage{latexsym}
\usepackage{exscale}
\usepackage{amsfonts}
\usepackage{graphicx}
\usepackage{mathrsfs}
\usepackage{amsmath,amscd,amsthm}
\usepackage{bbm,pifont}
\usepackage{enumerate}
\usepackage{color}
\usepackage{amssymb}
\usepackage{latexsym}
\usepackage{exscale}
\DeclareMathOperator*{\esssup}{ess\,sup}
\newcommand{\supp}{\operatorname{supp}}

\def\lz{\lambda}
\def\tbz{{\triangle_\lz}}

\begin{document}
\allowdisplaybreaks

\title[Weighted product Hardy spaces associated to operators]{Equivalence of Littlewood--Paley
square function and area function characterizations of weighted product\\
Hardy spaces associated to operators}
\author{Xuan Thinh Duong}
\address{Department of Mathematics, Macquarie University, Sydney, NSW 2109, Australia}
\email{xuan.duong@mq.edu.au}
\author{Guorong Hu}
\address{Department of Mathematics, Jiangxi Normal University, Nanchang, Jiangxi 330022, P. R. China}
\email{hugr@mail.ustc.edu.cn}
\author{Ji Li}
\address{Department of Mathematics, Macquarie University, Sydney, NSW 2109, Australia}
\email{ji.li@mq.edu.au}



\date{\today}

\begin{abstract}
Let $L_1$ and $L_2$ be non-negative self-adjoint operators acting on $L^2(X_1)$ and $L^2(X_2)$, respectively,
where $X_1$ and $X_2$ are spaces of homogeneous type. Assume that $L_1$ and $L_2$ have Gaussian heat kernel bounds.
This paper aims to study some equivalent characterizations of the weighted product Hardy spaces
$H^{p}_{w,L_{1},L_{2}}(X_{1}\times X_{2})$ associated to $L_{1}$ and $L_{2}$, for $p \in (0, \infty)$ and
the weight $w$ belongs to the product Muckenhoupt class $ A_{\infty}(X_{1} \times X_{2})$. Our main result is that the spaces $H^{p}_{w,L_{1},L_{2}}(X_{1}\times X_{2})$
introduced via area functions can be equivalently characterized by Littlewood--Paley $g$-functions,
Littlewood--Paley $g^{\ast}_{\lambda_{1}, \lambda_{2}}$-functions, and Peetre type maximal functions,
without any further assumptions beyond the Gaussian upper bounds on the heat kernels of $L_1$ and $L_2$.
Our results are new even in the unweighted product setting.
\end{abstract}

\maketitle

\section{Introduction}

The theory of Hardy spaces has been a successful story in modern harmonic analysis in the last fifty years.
In the classical case of the Euclidean space $\mathbb R^n$, it is well known that among other equivalent
characterizations the Hardy space $H^p(\mathbb R^n)$ are characterized by area functions, by Littlewood--Paley
$g$-functions and by atomic decomposition \cite{FS,St}. Concerning Hardy spaces $H^p(X)$  on a space of
homogeneous type $X$, a new approach to show the equivalence between characterizations of $H^p(X)$ by area functions
and $g$-functions is to use the Plancherel--Polya type inequality, which requires the H\"older continuity
and cancellation conditions \cite{DH}. About the more recent Hardy spaces $H^p_L(X)$ associated to an
operator $L$ on a space of homogeneous type $X$, one used to need extra assumptions to show that the
characterizations by area functions and by Littlewood-Paley $g$-functions are equivalent, for example,
H\"older continuity was assumed in  \cite{DLSY} and Moser type estimate in \cite{DLY}.
Only recently, the equivalence of the characterizations of $H_{L}^p(X)$ by area functions and by Littlewood--Paley
$g$-functions was obtained in  \cite{Hu} under no further assumption beyond the Gaussian heat kernel bounds.
Actually, the work in \cite{Hu} was done in the weighted setting.

The aim of the current paper is to prove the equivalence between the
characterizations of the weighted product Hardy spaces
$H^{p}_{w, L_{1},L_{2}}(X_{1}\times X_{2})$ in terms of the area funcions and
Littlewood--Paley square functions, see Theorems \ref{main1} and \ref{main2},
where we assume only that the operators $L_1$ and $L_2$ are non-negative self-adjoint and have
Gaussian upper bounds on their heat kernels.
This extends the main result in \cite{Hu} to the product setting.
The strength of our results is that not only they are new for the setting of product spaces
and covers larger classes of operators $L_1$ and $L_2$ but also
recover a number of known results whose proofs rely on extra regularity of the semigroups.
In particular, our Theorems \ref{main1} and \ref{main2}

(i) give a direct proof for the equivalent characterizations via Littlewood--Paley square functions of the classical product Hardy space
 by Chang--Fefferman  in \cite{CF},

(ii) provide a new proof of equivalent characterizations via Littlewood--Paley square functions of the product
Hardy spaces on spaces of homogeneous type in \cite{HLW} whose proofs required the H\"older continuity and cancellation condition,

(iii) provide the missing  characterizations of product Hardy spaces via Littlewood--Paley square functions in the setting
developed in \cite{DSTY} and \cite{DLY}, and

(iv) recover the recent related known
results in the setting of Bessel operators in \cite{DLWY} whose proofs relied on the H\"older regularity,
and results for Bessel Schr\"odinger operators in \cite{BDLWY} whose proofs used the Moser type inequality.

For more details and explanations of (iii) and (iv), we refer to Section 4.

We now recall some basic facts
concerning spaces of homogeneous type. Let $(X,\rho)$ be a  metric space, and $\mu$ be a positive Radon measure on $X$.
Write $V(x,r):=\mu(B(x,r))$, where $B(x,r)$ denotes the open ball centered at $x$ with radius $r$. We say that $(X,\rho, \mu)$ is
a space of homogeneous type if it satisfies the volume doubling property:
\begin{equation} \label{doub}
V(x,2r)  \leq V(x,r)
\end{equation}
for all $x \in X$ and $r >0$.  An immediate consequence of \eqref{doub} is that
there exist constants $C$ and $n$ such that
\begin{equation} \label{hom}
V(x, \lambda r) \leq C\lambda^{n}V(x,r)
\end{equation}
for all $x \in X$, $r>0$ and $\lambda \geq 1$.
The constant $n$ plays the role of an upper bound of the dimension, though it need not even be an integer,
and we want to take $n$ as small as possible.  There also exist constants $C$ and $D$,
$0\leq D \leq n$, so that
\begin{equation} \label{tos}
V(y,r)\leq C\left(1+\frac{\rho(x,y)}{r} \right)^{D}V(x,r)
\end{equation}
uniformly for all $x, y \in X$ and $r >0$. Indeed, property \eqref{tos} with $D=n$ is a direct consequence
of \eqref{hom}. In the case where $X$ is the Euclidean space $\mathbb{R}^{n}$ or a Lie group of polynomial
growth, $D$ can be chosen to be $0$.

Throughout this paper, we assume that, for $i=1,2$, $(X_{i}, \rho_{i}, \mu_{i})$ is a space of homogenous type
with $\mu(X_{i}) = \infty$. The constant $n$ (resp. $D$) in \eqref{hom} (resp. \eqref{tos}) for $(X_{i}, \rho_{i}, \mu_{i})$ is denoted by
$n_{i}$ (resp. $D_{i}$). Let $L_{i}$, $i=1,2$, be a linear operator on $L^{2}(X_{i},d\mu_{i})$
satisfying the following properties:

{\bf (H1)} Each $L_{i}$ is a non-negative self-adjoint operator on $L^{2}(X_{i},d\mu_{i})$;

{\bf (H2)} The kernel of the semigroup $e^{-t L_{i}}$, denoted by $p^{(i)}_{t}(x_{i}, y_{i})$, is a measurable
function on $X_{i} \times X_{i}$ and obeys a Gaussian upper bound, that is,
\begin{equation*} \label{GUB}
\left| p^{(i)}_{t}(x_{i},y_{i})\right| \leq \frac{C_{i}}{V(x_{i},\sqrt{t})}\exp \left(-\frac{\rho_{i}(x_{i},y_{i})^{2}}{c_{i}t} \right)
\end{equation*}
for all $t > 0$ and a.e. $(x_{i}, y_{i}) \in X_{i} \times X_{i}$,
where $C_i$ and $c_i$ are positive constants, for $i=1,2$.

\begin{definition}
Let $\Phi_{1}, \Phi_{2} \in \mathcal{S}(\mathbb{R})$.

${\rm a)}$ Given a function $f \in L^{2}(X_{1} \times X_{2})$, we define the product type Littlewood--Paley
$ g$-function $g_{\Phi_{1},\Phi_{2},L_{1},L_{2}}(f)$ associated to
$L_{1}$ and $L_{2}$ by
\begin{equation}
g_{\Phi_{1},\Phi_{2},L_{1},L_{2}}(f)(x_{1},x_{2}):= \left(\int_{0}^{\infty}\int_{0}^{\infty}\left|\Phi_{1}(t_{1}\sqrt{L_{1}})
\otimes \Phi_{2}(t_{2}\sqrt{L_{2}})
f(x_{1},x_{2}) \right|^{2}  \frac{dt_{1}}{t_{1}} \frac{dt_{2}}{t_{2}} \right)^{1/2} .
\end{equation}

${\rm b)}$  The product type area function $S_{\Phi_{1},\Phi_{2},L_{1},L_{2}}(f)$ associated to
$L_{1}$ and $L_{2}$ is defined by
\begin{equation}
\begin{split}
&S_{\Phi_{1},\Phi_{2},L_{1},L_{2}}(f)(x_{1},x_{2})\\
&:= \left(\iint_{\Gamma_{1}(x_{1}) \times \Gamma_{2}(x_{2})}\left|\Phi_{1}(t_{1}\sqrt{L_{1}})\otimes \Phi_{2}(t_{2}\sqrt{L_{2}})
f(y_{1},y_{2}) \right|^{2}  \frac{d\mu_{1}(y_{1})dt_{1}}{V(x_{1},t_{1})t_{1}} \frac{d\mu_{2}(y_{2})dt_{2}}{V(x_{2},t_{2})t_{2}} \right)^{1/2},
\end{split}
\end{equation}
where $\Gamma_{i}(x_{i}):= \{(y_{i},t_{i}) \in X_{i} \times (0, \infty): \rho_{i}(x_{i},y_{i})<t_{i}\}$ for $i=1,2$.

${\rm c)}$  For $\lambda_{1}, \lambda_{2}, t_{1}, t_{2} >0$, the product Peetre type maximal functions associated to $L_{1}$ and $L_{2}$ is defined by
\begin{equation} \label{Peetre}
\begin{split}
&\big[\Phi_{1}(t_{1}\sqrt{L_{1}})\otimes \Phi_{2}(t_{2}\sqrt{L_{2}})
\big]^{\ast}_{\lambda_{1},\lambda_{2}}f(x_{1}, x_{2}) \\
&\quad\quad : = \esssup_{(y_{1}, y_{2}) \in X_{1}\times X_{2}}
\frac{\big|\Phi_{1}(t_{1}\sqrt{L_{1}})\otimes \Phi_{2}(t_{2}\sqrt{L_{2}})
f(y_{1},y_{2})\big|}{(1+ t_{1}^{-1} \rho_{1}(x_{1},y_{1}))^{\lambda_{1}} (1+t_{2}^{-1}\rho_{2}(x_{2}, y_{2})  )^{\lambda_{2}}},
\quad (x_{1}, x_{2}) \in X_{1} \times X_{2}.
\end{split}
\end{equation}

${\rm d)}$  The product type Littlewood--Paley $g_{\lambda_{1}, \lambda_{2}}^{\ast}$-function associated to $L_{1}$
and $L_{2}$ is defined by
\begin{align}
&g^{\ast}_{\Phi_{1},\Phi_{2},,L_{1},L_{2},\lambda_{1},\lambda_{2}}(f)(x_{1},x_{2}) \\
&:=\!\left(\int_{0}^{\infty}\!\!\!\int_{0}^{\infty}\!\!\!\int_{X_{1}}\!\!\int_{X_{2}}\frac{\left|\Phi_{1}(t_{1}\sqrt{L_{1}})\otimes \Phi_{2}(t_{2}\sqrt{L_{2}})
f(y_{1},y_{2}) \right|^{2}}{(1+t_{1}^{-1}\rho_{1}(x_{1},y_{1}))^{n_{1}\lambda_{1}}(1+t_{2}^{-1}\rho_{2}(x_{2},y_{2}))^{n_{2}\lambda_{2}}}
 \frac{d\mu_{1}(y_{1})dt_{1}}{V(x_{1},t_{1})t_{1}} \frac{d\mu_{2}(y_{2})dt_{2}}{V(x_{2},t_{2})t_{2}} \right)^{1\over2}.\nonumber
\end{align}
\end{definition}

Following \cite{Feff,GCRF}, we introduce product Muckenhoupt weights on spaces of homogeneous type.
\begin{definition}
A non-negative locally integrable function $w$ on $X_{1} \times X_{2}$ is said to belong to the product
Muckenhoupt class $A_{p}(X_{1} \times X_{2})$ for a given $p \in (1, \infty)$, if there is a constant $C$
such that for all balls $B_{1} \subset X_{1}$ and $B_{2} \subset X_{2}$,
\begin{align*}
&\left( \frac{1}{\mu_{1}(B_{1}) \mu_{2}(B_{2})} \iint_{B_{1} \times B_{2}} w(x_{1},x_{2})d\mu_{1} (x_{1})d\mu_{2}(x_{2})\right) \\
& \quad \times
\left(\frac{1}{\mu_{1} (B_{1})\mu_{2}( B_{2})} \iint_{ B_{1} \times B_{2}  } w(x_{1},x_{2})^{-1/(p-1)} d\mu_{1} (x_{1})d\mu_{2}(x_{2})
\right)^{p-1} \leq C.
\end{align*}
The class $A_{1}(X_{1} \times X_{2})$ is defined to be the collection of all non-negative locally integrable
functions $w$ on $X_{1} \times X_{2}$ such that
\begin{equation*}
\left( \frac{1}{\mu_{1}(B_{1}) \mu_{2}(B_{2})} \iint_{B_{1} \times B_{2}} w(x_{1},x_{2})d\mu_{1} (x_{1})d\mu_{2}(x_{2}) \right)
\|w^{-1}\|_{L^{\infty}(B_{1} \times B_{2})} \leq C.
\end{equation*}
for all balls $B_{1} \subset X_{1}$ and $B_{2} \subset X_{2}$.
\end{definition}

We let $A_{\infty}(X_{1} \times X_{2}) : = \cup_{1 \leq p < \infty}A_{p}(X_{1}
\times X_{2})$ and, for any $w  \in A_{\infty}(X_{1} \times X_{2})$, define
\begin{equation*}
q_{w} := \inf \left\{q \in [1, \infty): w \in A_{q}(X_{1} \times X_{2})\right\},
\end{equation*}
the critical index for $w$ (see, for instance, \cite{GCRF}). For $1 < p <\infty$, the weighted Lebesgue space $L^{p}_{w}(X_{1} \times X_{2})$
is defined to be the collection of all measurable functions $f$ on $X_{1} \times X_{2}$ for which
\begin{equation*}
\|f\|_{L_{w}^{p}(X_{1}\times X_{2})} : = \left( \iint_{X_{1}\times X_{2}} |f(x_{1}, x_{2})|^{p}
w(x_1,x_2)d\mu_{1}(x_{1})d\mu_{2}(x_{2})\right)^{1/p}<\infty.
\end{equation*}

We next introduce a class of functions on $\mathbb{R}$ which will play a significant role in our formulation.

\begin{definition}
A function $\Phi \in \mathcal{S}(\mathbb{R})$ is said to belong to the class
$\mathcal{A}(\mathbb{R})$ if it satisfies the Tauberian condition, namely,
\begin{equation} \label{tauberian}
|\Phi(\lambda)| >0 \quad \mbox{on  } \{\varepsilon /2 < |\lambda| <2\varepsilon\}
\end{equation}
for some $\varepsilon >0$.
\end{definition}

Now we are ready to state our main results.

\begin{theorem} \label{main1}
Let $\Phi_{1}, \Phi_{2}, \widetilde{\Phi}_{1}, \widetilde{\Phi}_{2} \in \mathcal{A}(\mathbb{R})$ be even functions satisfying
\begin{align*}
\Phi_{1}(0)= \Phi_{2}(0) =\widetilde{\Phi}_{1}(0)= \widetilde{\Phi}_{2}(0) =0.
\end{align*}
Let $p \in (0, \infty)$ and $w \in A_{\infty}(X_{1} \times X_{2})$. Then there exists
a constant $C =C (p, w, \Phi_{1}, \Phi_{2}, \widetilde{\Phi}_{1}, \widetilde{\Phi}_{2})$ such that for all $f \in L^{2}(X_{1} \times X_{2})$,
\begin{equation*}
C^{-1} \|g_{\widetilde{\Phi}_{1},\widetilde{\Phi}_{2},L_{1},L_{2}}(f)\|_{L_{w}^{p}(X_{1}\times X_{2})} \leq \|g_{\Phi_{1},\Phi_{2},L_{1},L_{2}}(f)\|_{L_{w}^{p}(X_{1}\times X_{2}) }
\leq C \|g_{\widetilde{\Phi}_{1},\widetilde{\Phi}_{2},L_{1},L_{2}}(f)\|_{L_{w}^{p}(X_{1}\times X_{2})}.
\end{equation*}
\end{theorem}

\begin{theorem} \label{main2}
Let $\Phi_{1}, \Phi_{2} \in \mathcal{A}(\mathbb{R})$ be even functions.
Let $p \in (0, \infty)$, $\lambda_{i} >\frac{2q_{w}}{\min\{p,2\}}$
and $\lambda_{i}' > \frac{(n_{i}+D_{i})q_{w}}{\min\{p,2\}}$, $i=1,2$. Then for $f \in L^{2}(X_{1} \times X_{2})$
we have the following (quasi-)norm equivalence:
\begin{align} \label{man}
\|S_{\Phi_{1},\Phi_{2}, L_{1},L_{2}}(f)\|_{L^{p}_{w}(X_{1}\times X_{2})} \sim \|g_{\Phi_{1},\Phi_{2},L_{1},L_{2}}(f)\|_{L^{p}_{w}(X_{1}\times X_{2}) }
 \sim \|g^{\ast}_{\Phi_{1},\Phi_{2},L_{1},L_{2},\lambda_{1},\lambda_{2}}(f) \|_{L^{p}_w(X_{1}\times X_{2})} \nonumber \\
  \sim  \left\|\left(\int_{0}^{\infty} \int_{0}^{\infty}\left|\big[\Phi_{1}(t_{1}\sqrt{L_{1}})
\otimes \Phi_{2}(t_{2}\sqrt{L_{2}})\big]^{\ast}_{\lambda_{1}',\lambda_{2}'}f\right|^{2}\frac{dt_{1}}{t_{1}}\frac{dt_{2}}{t_{2}}
\right)^{1/2}  \right\|_{L_{w}^{p}(X_{1}\times X_{2})}.
\end{align}
\end{theorem}

Having these results, one can introduce weighted product Hardy spaces associated to $L_{1}$ and $L_{2}$ as follows:
\begin{definition}
Let $p \in (0, \infty)$, $w \in A_{\infty}(X_{1} \times X_{2})$,
and $\Phi_{1}, \Phi_{2} \in \mathcal{A}(\mathbb{R})$ be even functions satisfying
\begin{align*}
\Phi_{1}(0)= \Phi_{2}(0) =0.
\end{align*}
The weighted product Hardy space $H^{p}_{w,L_{1},L_{2}}(X_{1}\times X_{2})$ associated to $L_{1}$ and $L_{2}$
is defined to be the completion of the set
\begin{align*}
\left\{f \in L^{2}(X_{1} \times X_{2}): S_{\Phi_{1},\Phi_{2},L_{1},L_{2}}(f) \in L_{w}^{p}(X_{1}\times X_{2})\right\}
\end{align*}
with respect to the (quasi-)norm
\begin{equation*}
\|f\|_{H^{p}_{w,L_{1},L_{2}}(X_{1}\times X_{2}) } : =\|S_{\Phi_{1},\Phi_{2},L_{1},L_{2}}(f)\|_{L^{p}_{w}(X_{1}\times X_{2})}.
\end{equation*}
\end{definition}

\begin{remark}\label{remark}
Combining Theorems \ref{main1} and \ref{main2} we see that the definition of $H^{p}_{w,L_{1},L_{2}}(X_{1}\times X_{2})$ is
independent of the choice of the even functions $\Phi_{1}, \Phi_{2}$,
as long as $\Phi_{1}, \Phi_{2} \in \mathcal{A}(\mathbb{R})$ and satisfy $\Phi_{1}(0)= \Phi_{2}(0) =0$.
In particular, if we choose $\Phi_{1}(\lambda)=\Phi_{2}(\lambda) = \lambda^{2} e^{-\lambda^{2}}$ for
$\lambda \in \mathbb{R}$, then the (quasi-)norm of $H^{p}_{w,L_{1},L_{2}}(X_{1}\times X_{2})$ can be written as
\begin{align*}
&\|f\|_{H^{p}_{w, L_{1},L_{2}}(X_{1}\times X_{2}) } \\
&: = \left\|\left(\iint_{\Gamma_{1}(x_{1}) \times \Gamma_{2}(x_{2})}\left|(t_{1}^{2}L_{1}e^{-t_{1}^{2}L_{1}})\otimes
(t_{2}^{2}L_{2}e^{-t_{2}^{2}L_{2}}) f(y_{1},y_{2}) \right|^{2}  \frac{d\mu_{1}(y_{1})dt_{1}}{V(x_{1},t_{1})t_{1}} \frac{d\mu_{2}(y_{2})dt_{2}}{V(x_{2},t_{2})t_{2}} \right)^{1\over2} \right\|_{L^{p}_{w}(X_{1}\times X_{2})}.
\end{align*}
Furthermore, from Theorem \ref{main2} we see that each quantity in \eqref{man}
can be used as an equivalent (quasi-)norm of the space $H^{p}_{w,L_{1},L_{2}}(X_{1}\times X_{2})$.
\end{remark}

As mentioned above, we make no further assumption on the heat kernel of $L_{1}$ or $L_{2}$ beyond the Gaussian
upper bounds. Thus, the approach in \cite{DLSY} which uses a Plancherel-Polya type inequality
and the approach in \cite{DLY} which uses a discrete characterization can not be applied directly to our setting.
To achieve our goal, we will follow the approach in \cite{BMP1,BMP2, Rychkov}, whose key ingradient
is a sub-mean value property; see Lemma \ref{central} below. This approach has recently been used in
\cite{Hu} to derive the equivalence of Littlewood--Paley $g$-function and area function characterisations
of one-parameter Hardy spaces associated to operators. However, the Littlewood--Paley
$g$-function and area function in \cite{Hu} are only defined via the heat semigroup, which are less general than
those defined in the current paper.

\section{Preliminaries}

In this section we collect some facts and technical results which will be needed in the subsequent section.
We start by noting that, if $(X, \rho, \mu)$ is a space of homogeneous type,
then for any $N >n$, there exists a constant $C=C(N)$ such that
\begin{equation}\label{inte}
\int_{X} \left(1+\frac{\rho(x,y)}{t}\right)^{-N} d\mu(y) \leq CV(x,t)
\end{equation}
for all $x\in X$ and $t>0$.

\begin{lemma} \label{smooth}
Assume that $(X,\rho,\mu)$ is a space of homogeneous type and $L$
is a non-negative self-adjoint operator on $L^{2}(X,d\mu)$ whose heat kernel obeys the Gaussian
upper bound. Let $\Phi \in \mathcal{S}(\mathbb{R})$ be even functions.
Then for every $N>0$, there exists a constant $C=C(\Phi,N)$ such that the kernel
$K_{\Phi(t\sqrt{L})}(x,y)$ of the operator $\Phi(t\sqrt{L})$ satisfies
\begin{equation*}
\big| K_{\Phi(t\sqrt{L})}(x,y)\big|
 \leq \frac{C}{V(x,t)}\left(1+ \frac{\rho(x,y)}{t}\right)^{-N}.
\end{equation*}
\end{lemma}
\begin{proof}
For the proof, we refer to \cite[Lemma 2.3]{CD}. See also \cite[Lemma 2.1]{SY}.
\end{proof}

\begin{lemma} \label{ODE}
Assume that $(X,\rho,\mu)$ is a space of homogeneous type and $L$
is a non-negative self-adjoint operator on $L^{2}(X,d\mu)$ whose heat kernel obeys the Gaussian
upper bound. Let $\Phi, \Psi \in \mathcal{S}(\mathbb{R})$ be even functions and let $\Psi$ satisfy
\begin{equation} \label{704}
\Psi^{(\nu)}(0)=0, \quad \nu=0,1,\cdots, m
\end{equation}
for some positive odd integer $m$. Then for every $N>0$, there exists a constant $C=C(\Phi, \Psi, N, m)$ such
that for all $s\geq t>0$,
\begin{align} \label{sirre}
\left|K_{\Phi(s\sqrt{L})\Psi(t\sqrt{L})}(x,y)\right|
\leq C \left(\frac{t}{s}\right)^{m+1}\frac{1}{V(x,s)}\left(1+\frac{\rho(x,y)}{s}\right)^{-N}.
\end{align}
\end{lemma}

\begin{proof}
First note that the property \eqref{704} implies
that the function $\lambda \mapsto \lambda^{-(m+1)} \Psi(\lambda)$ is an even function, smooth at $0$, and belongs
to $\mathcal{S}(\mathbb{R})$. We set $\Phi_{m}(\lambda) := \lambda^{m+1} \Phi(\lambda)$ and
$\Psi_{m}(\lambda): = \lambda^{-(m+1)} \Psi(\lambda)$ for $\lambda \in \mathbb{R}$.
Then both $\Phi_{m}$ and $\Psi_{m}$ are even functions and belong to $\mathcal{S}(\mathbb{R})$.
Since
\begin{align*}
\Phi(s\sqrt{L})\Psi(t\sqrt{L}) &= \left(\frac{t}{s} \right)^{m +1}\big[ (s\sqrt{L})^{m +1}
\Phi(s\sqrt{L})\big]\big[(t\sqrt{L})^{-(m+1)}\Psi(t\sqrt{L})\big] \\
&= \left(\frac{t}{s} \right)^{m +1}\Phi_{m}(s\sqrt{L})\Psi_{m}(t\sqrt{L}),
\end{align*}
it follows from Lemma \ref{smooth} that
\begin{align} \label{fdaga}
& \left|K_{\Phi(s\sqrt{L})\Psi(t\sqrt{L})} (x,y) \right|
=\left(\frac{t}{s} \right)^{m +1}\left| K_{\Phi_{m}(s\sqrt{L})\Psi_{m}(t\sqrt{L})} (x,y)\right| \\
&\quad  \leq \left(\frac{t}{s} \right)^{m +1} \int_{X} \left|K_{ \Phi_{m}(s\sqrt{L})}(x,z) K_{\Psi_{m}(t\sqrt{L})}(z,y) \right|d\mu(z)\nonumber \\
 \leq C(\Phi, \Psi, N, m) &\left(\frac{t}{s} \right)^{m +1}  \int_{X}
\frac{1}{V(x,s)}\left( 1+\frac{\rho(x,z)}{s} \right)^{-N}\frac{1}{V(y,t)}\left( 1+\frac{\rho(z,y)}{t} \right)^{-(N+n+1)}d\mu(z). \nonumber
\end{align}
For $s \geq t >0$, we have
\begin{align*}
\left( 1+\frac{\rho(x,z)}{s} \right)^{-N}\left( 1+\frac{\rho(z,y)}{t} \right)^{-N} \leq \left( 1+\frac{\rho(x,y)}{s} \right)^{-N}.
\end{align*}
This along with \eqref{inte} yields
\begin{equation} \label{fdaga2}
\begin{split}
\int_{X} & \left( 1+\frac{\rho(x,z)}{s} \right)^{-N}\left( 1+\frac{\rho(z,y)}{t} \right)^{-(N+n+1)}d\mu(z) \\
& \leq \left( 1+\frac{\rho(x,y)}{s} \right)^{-N} \int_{X} \left( 1+\frac{\rho(z,y)}{t} \right)^{-(n+1)}d\mu(z) \\
& \leq C\left( 1+\frac{\rho(x,y)}{s} \right)^{-N}  V(y,t).
\end{split}
\end{equation}
Combining \eqref{fdaga} and \eqref{fdaga2} we obtain \eqref{sirre}.
\end{proof}

\begin{lemma} \label{uuuuu}
Suppose $\Phi \in \mathcal{A}(\mathbb{R})$ is an even function.
Then there exist even functions $\Psi, \Upsilon, \Theta \in \mathcal{S}(\mathbb{R})$
such that
\begin{align*}
&\quad\quad \supp \Upsilon \subset \{|\lambda| \leq 2\varepsilon\}, \\
&\quad \supp \Theta \subset \{\varepsilon/2 \leq |\lambda|\leq 2\varepsilon\}
\end{align*}
and
\begin{equation*}
\Psi(\lambda)\Upsilon(\lambda)+ \sum_{k =1}^{\infty}\Phi(2^{-2k}\lambda)
\Theta(2^{-2k}\lambda)=1 \quad \mbox{for all } \lambda \in \mathbb{R},
\end{equation*}
where $\varepsilon$ is a constant from \eqref{tauberian}.
\end{lemma}
\begin{proof}
Define $\Psi(\lambda):=e^{-\lambda^{2}}$, $\lambda \in \mathbb{R}$. Obviously, $\Psi \in \mathcal{S}(\mathbb{R})$ and $\Psi$ is even.
Choose nonnegative even functions $\Omega, \Gamma \in \mathcal{S}(\mathbb{R})$ such that
\begin{align*}
&\Omega(\lambda) \neq 0  \Longleftrightarrow |\lambda| < 2\varepsilon,\\
\Gamma(&\lambda) \neq 0 \Longleftrightarrow \varepsilon/2 < |\lambda| < 2\varepsilon.
\end{align*}
Then we set
\begin{equation} \label{ser}
\Xi(\lambda):=\Psi(\lambda)\Omega(\lambda)+ \sum_{k=1}^{\infty}
\Phi(2^{-k}\lambda)\Gamma(2^{-k}\lambda), \quad \lambda \in \mathbb{R}.
\end{equation}
From the properties of $\Phi, \Psi, \Omega$ and $\Gamma$ it follows that $\Xi$ is strictly positive on $\mathbb{R}$.
In addition, from the properties of $\Omega$ and $\Gamma$ we see that for any fixed $\lambda_{0} \in \mathbb{R}\backslash \{0\}$,
the number of those $k$'s for which $\Phi(2^{-k}\lambda)\Gamma(2^{-k}\lambda)$ do not vanish identically in $(\frac{4\lambda_{0}}{5},
\frac{6\lambda_{0}}{5})$ is no more than 4, which implies that $\Xi$ is smooth in $(\frac{4\lambda_{0}}{5}, \frac{6\lambda_{0}}{5})$
and hence $\Xi \in C^{\infty}(\mathbb{R}\backslash \{0\})$. It is obvious that $\Xi$ is also smooth at the origin $0$. Therefore $\Xi \in C^{\infty}(\mathbb{R})$. Now define the functions $\Upsilon$ and $\Theta$  respectively by
\begin{equation*}
\Upsilon(\lambda):=\frac{\Omega(\lambda)}{\Xi(\lambda)}
\quad\mbox{and}\quad\Theta(\lambda):=\frac{\Gamma(\lambda)}{\Xi(\lambda) }.
\end{equation*}
Then it is straightforward to verify that $\Psi,\Upsilon$ and $\Theta$ satisfy the desired properties.
\end{proof}

\begin{lemma} \label{uuuuu1}
Suppose $\Phi \in \mathcal{A}(\mathbb{R})$ is an even function.
Then there exists an even functions $\Theta \in \mathcal{S}(\mathbb{R})$ such that
\begin{align*}
\supp \Theta \subset \{\varepsilon/2 \leq |\lambda|\leq 2\varepsilon\}
\end{align*}
and
\begin{equation*}
\sum_{k =-\infty}^{\infty}\Phi(2^{-k}\lambda)
\Theta(2^{-k}\lambda)=1 \quad \mbox{for all } \lambda \in \mathbb{R} \backslash \{0\},
\end{equation*}
where $\varepsilon$ is a constant from \eqref{tauberian}.
\end{lemma}
\begin{proof}
The proof is analogous to that of Lemma \ref{uuuuu} and thus we omit the details.
\end{proof}

\begin{lemma} \label{negle}
Assume that $(X,\rho,\mu)$ is a space of homogeneous type with $\mu(X) =\infty$ and $L$
is a non-negative self-adjoint operator on $L^{2}(X,d\mu)$ whose heat kernel obeys the Gaussian
upper bound. Let $\{E(\lambda): \lambda \geq 0\}$ be spectral resolution of $L$.
Then the spectral measure of the set $\{0\}$ is zero, i.e., the point $\lambda = 0$ may be neglected in the spectral resolution.
\end{lemma}
\begin{proof}
Assume by contradiction that  $E(\{0\}) \neq 0$, then there exists $g \in L^{2}(X)$
such that $f := E(\{0\})g$ is not the zero element in $L^{2}(X,d\mu)$. Since $E(\{0\})$
is a an orthogonal projection,
\begin{equation*}
E(\{0\})f = E(\{0\})E(\{0\})g = E(\{0\})g = f.
\end{equation*}
It follows that for all $t>0$,
\begin{align*}
e^{-tL}f =\int_{0}^{\infty}e^{-t\lambda}dE(\lambda)f =\int_{0}^{\infty}
e^{-t\lambda}dE(\lambda)E(\{0\})f =\int_{\{0\}}e^{-t\lambda}dE(\lambda)f = E(\{0\})f = f.
\end{align*}
Hence, for a.e. $x \in X$ and all $t >0$, we have
\begin{align*}
&|f(x)|  = \big|e^{-tL}f(x) \big|
\leq \int_{X} |p_{t}(x,y)||f(y)|d\mu(y) \\
&\leq  \|f\|_{L^{2}(X,d\mu)} \left(\int_{X}|p_{t}(x,y)|^{2} d\mu(y)\right)^{1/2}\\
& \leq  C\|f\|_{L^{2}(X,d\mu)} \left(\int_{X}\frac{1}{V(x,\sqrt{t})^{2}}
\left(1+\frac{\rho(x,y)}{\sqrt{t}}\right)^{-(n+1)} d\mu(y)\right)^{1/2}\\
& \leq  C\|f\|_{L^{2}(X,d\mu)}V(x,\sqrt{t})^{-1/2}.
\end{align*}
Since $\mu(X) =\infty$, letting $t \rightarrow \infty$ in the above yields that $f(x) =0$.
Hence $f =0$ in $L^{2}(X,d\mu)$, which leads to a contradiction. Therefore we must have $E(\{0\})=0$.
\end{proof}

The following two lemmas are two-parameter counterparts of Lemma 2 and Lemma 3 in \cite{Rychkov}, respectively.
These can be proved by slightly modifying the proofs of the corresponding one-parameter results. We omit the details here.

\begin{lemma} \label{RY} {\rm (\cite[Lemma 2]{Rychkov})}
Let $0<p,q < \infty$ and $\sigma_{1},\sigma_{2} >0$. Let $w$ be arbitrary weight (i.e., non-negative locally integrable function)
on $X_{1} \times X_{2}$. Let $\{g_{j_{1},j_{2}}\}_{j_{1},j_{2}=-\infty}^{\infty}$
be a sequence of non-negative measurable functions on $X_{1}\times X_{2}$ and put
\begin{equation} \label{sequ}
h_{j_{1},j_{2}}(x_{1},x_{2}) =\sum_{k_{1} =-\infty}^{\infty}\sum_{k_{2} =-\infty}^{\infty}2^{-|k_{1}-j_{1}|\sigma_{1}}2^{-|k_{2}-j_{2}|\sigma_{2}}
g_{k_{1},k_{2}}(x_{1},x_{2})
\end{equation}
for $(x_{1},x_{2}) \in X_{1}\times X_{2}$ and $j_{1}, j_{2}\in \mathbb{Z}$.
Then, there exists a constant $C=C(q,\sigma_{1},\sigma_{2})$ such that
\begin{equation*}
\left\|\big\{ h_{j_{1},j_{2}}\big\}_{j_{1},j_{2}=-\infty}^{\infty}\right\|_{L_{w}^{p}(\ell^{q})}
\leq C\left\|\big\{ g_{j_{1}, j_{2}}\big\}_{j_{1},j_{2}=-\infty}^{\infty}\right\|_{L_{w}^{p}(\ell^{q})},
\end{equation*}
where
\begin{align} \label{shem}
&\left\|\big\{ g_{j_{1}, j_{2}}\big\}_{j_{1},j_{2}=-\infty}^{\infty}\right\|_{L_{w}^{p}(\ell^{q})}
:=\left\|\left\|\big\{ g_{j_{1}, j_{2}}\big\}_{j_{1},j_{2}=-\infty}^{\infty} \right\|_{\ell_{q}}\right\|_{L_{w}^{p}(X_{1} \times X_{2})}   \\
&\quad\quad
= \left\|\left(\sum_{j_{1}=-\infty}^{\infty} \sum_{j_{2}=-\infty}^{\infty}|g_{j_{1},j_{2}}(x_{1},x_{2})|^{q}\right)^{1/q}
\right\|_{L_{w}^{p}(X_{1}\times X_{2})}. \nonumber
\end{align}
\end{lemma}

\begin{lemma} \label{lpm} {\rm (\cite[Lemma 3]{Rychkov})}
Let $0<r \leq 1$, and let $\{b_{j_{1},j_{2}}\}_{j_{1},j_{2}=-\infty}^{\infty}$ and $\{d_{j_{1},j_{2}}\}_{j_{1},j_{2} =-\infty}^{\infty}$
be two sequences taking values in $(0,\infty]$ and $(0, \infty)$ respectively. Assume that there exists
$N_{0} >0$ such that
\begin{align} \label{267}
d_{j_{1},j_{2}} = O(2^{j_{1}N_{0}}2^{j_{2}N_{0}}), \quad j_{1}, j_{2} \rightarrow \infty,
\end{align}
and that for every $N >0$ there exists a finite constant $C=C_{N}$ such that
\begin{equation} \label{984}
d_{j_{1},j_{2}} \leq C_{N} \sum_{k_{1}=j_{1}}^{\infty}\sum_{k_{2}=j_{2}}^{\infty}2^{(j_{1}-k_{1})N}
2^{(j_{2}-k_{2})N}b_{k_{1},k_{2}}d_{k_{1}, k_{2}}^{1-r}, \quad j_{1},j_{2} \in \mathbb{Z}.
\end{equation}
Then for every $N>0$,
\begin{equation} \label{176}
d_{j_{1}, j_{2}}^{r} \leq C_{N} \sum_{k_{1}=j_{1}}^{\infty}\sum_{k_{2}= j_{2}}^{\infty}2^{(j_{1}-k_{1})Nr}
2^{(j_{2}- k_{2})Nr}b_{k_{1},k_{2}}, \quad j_{1}, j_{2} \in \mathbb{Z},
\end{equation}
with the same constants $C_{N}$.
\end{lemma}

For a locally integrable function $f$ on $X_{1} \times X_{2}$, the strong maximal function is defined by
\begin{equation*}
\mathcal{M}_{s}(f)(x_{1},x_{2}) := \sup_{(x_{1}, x_{2}) \in B_{1} \times B_{2}} \frac{1}{\mu_{1}(B_{1})
\mu_{2}(B_{2})} \iint_{B_{1} \times B_{2} } |f(y_{1},y_{2})|d\mu_{1}(y_{1})d\mu_{2}(y_{2}),
\end{equation*}
where $B_{i}$ runs over all balls in $X_{i}$, $i=1,2$.
Using \eqref{tos} and the volume doubling property, one can easily show that if $N_{i} > n_{i}+D_{i}$ for $i=1,2$, then
\begin{equation} \label{smfx}
\iint_{X_{1} \times X_{2}} \frac{|f(y_{1},y_{2})|}{\prod_{i=1}^{2}V(y_{i},t_{i})
(1 + t_{i}^{-1}\rho_{i}(x_{i},y_{i}))^{N_{i}}}d\mu_{1}(y_{1})d\mu_{2}(y_{2})
\leq C\mathcal{M}_{s}(f)(x_{1},x_{2}).
\end{equation}
We will also need the following weighted vector-valued inequality for strong maximal functions on spaces of homogeneous type.
See, for instance, \cite{GCRF} and \cite{Sato}.

\begin{lemma} \label{maximal}
Suppose $1<p<\infty$, $1<q \leq \infty$ and $w \in A_{p}(X_{1} \times X_{2})$.
Then there exists a constant $C$ such that
\begin{equation*}
\left\|\big\{\mathcal{M}_{s}(f_{j_{1}, j_{2}})\big\}_{j_{1},j_{2}=-\infty}^{\infty}\right\|_{L_{w}^{p}(\ell^{q})} \leq C
\left\|\big\{f_{j_{1}, j_{2}}\big\}_{j_{1},j_{2}=-\infty}^{\infty}\right\|_{L_{w}^{p}(\ell^{q})}
\end{equation*}
for all sequences $\big\{ f_{j_{1}, j_{2}}\big\}_{j_{1},j_{2}=-\infty}^{\infty}$ on $X_{1} \times X_{2}$,
where the space $L_{w}^{p}(\ell^{q})$ is defined by \eqref{shem}.
\end{lemma}

\section{Proofs of Theorems \ref{main1} and \ref{main2}}

We divide the proof of Theorems \ref{main1} and \ref{main2} into a sequence of lemmas.

\begin{lemma} \label{lkjhh}
Let $\Phi_{1}, \Phi_{2} \in \mathcal{S}(\mathbb{R})$ be even functions.  Let $p \in (0, \infty)$, $w \in A_{\infty}(X_{1} \times X_{2})$,
and $\lambda_{1},\lambda_{2} >\frac{2q_{w}}{\min\{p,2\}}$. Then there exists a constant $C$
such that for all $f \in L^{2}(X_{1}\times X_{2})$,
\begin{align*}
\big\|g^{\ast}_{\Phi_{1},\Phi_{2},L_{1},L_{2},\lambda_{1},\lambda_{2}}(f)\big\|_{L_{w}^{p}(X_{1}\times X_{2})}
\leq C \big\|S_{\Phi_{1},\Phi_{2},L_{1},L_{2}}(f)\big\|_{L_{w}^{p}(X_{1}\times X_{2})}.
\end{align*}
\end{lemma}
\begin{proof}
This can be proved by a standard argument; see, for instance, \cite[Theorem 4 in Ch. 4]{ST}. We omit the details here.
\end{proof}

\begin{lemma} \label{8787}
Let $\Phi_{1},\Phi_{2} \in \mathcal{S}(\mathbb{R})$ be even functions.
Let $p \in (0, \infty)$, $\lambda_{1}, \lambda_{2} >0$, and $w$ be arbitrary
weight (i.e., non-negative locally integrable function) on $X_{1} \times X_{2}$.
Then there exists a constant $C$ such that for all  $f \in L^{2}(X_{1} \times X_{2})$,
\begin{align*}
\|S_{\Phi_{1},\Phi_{2},L_{1},L_{2}}(f)\|_{L_{w}^{p}(X_{1}\times X_{2})} \leq C\left\|\left(\int_{0}^{\infty} \int_{0}^{\infty}\left|\big[\Phi_{1}(t_{1}\sqrt{L_{1}})\otimes \Phi_{2}(t_{2}\sqrt{L_{2}})
\big]^{\ast}_{\lambda_{1},\lambda_{2}}f\right|^{2}\frac{dt_{1}}{t_{1}}\frac{dt_{2}}{t_{2}}
\right)^{1/2}\right\|_{L_{w}^{p}(X_{1} \times X_{2})}.
\end{align*}
\end{lemma}
\begin{proof}
Observe that for all $\lambda_{1}, \lambda_{2}, t_{1}, t_{2} >0$ and all $(x_{1}, x_{2})\in X_{1} \times X_{2}$,
\begin{align*}
&\frac{1}{V(x_{1}, t_{1})V(x_{2},t_{2})}\iint_{B(x_{1},t_{1})\times B(x_{2},t_{2})}\left|\Phi_{1}( t_{1}\sqrt{L_{1}})
\otimes \Phi_{2}( t_{2} \sqrt{L_{2}})f(y_{1},y_{2})\right|^{2}d\mu_{1}(y_{1})d\mu_{2}(y_{2})\\
&\quad\quad\quad\quad\quad\quad \leq \esssup_{(y_{1}, y_{2}) \in B(x_{1},t_{1}) \times B(x_{2},t_{2})}\left|\Phi_{1}( t_{1}\sqrt{L_{1}})
\otimes \Phi_{2}( t_{2}\sqrt{L_{2}}) f(y_{1},y_{2})\right|^{2} \\
& \quad\quad\quad\quad\quad\quad \leq 2^{2\lambda_{1}}2^{2\lambda_{2}}\left|\big[\Phi_{1}( t_{1}\sqrt{L_{1}})
\otimes \Phi_{2}( t_{2}\sqrt{L_{2}})\big]^{\ast}_{\lambda_{1},\lambda_{2}} f(x_{1},x_{2})\right|^{2}.
\end{align*}
Taking the norm $\int_{0}^{\infty}\int_{0}^{\infty}|\cdot| \frac{dt_{1}}{t_{1}}\frac{dt_{2}}{t_{2}}$ on both sides gives
the pointwise estimate
\begin{align*}
\big[S_{\Phi_{1},\Phi_{2}, L_{1}, L_{2}}(f)(x_{1},x_{2})\big]^{2}
\leq 2^{2\lambda_{1}}2^{2\lambda_{2}} \int_{0}^{\infty} \int_{0}^{\infty}
\left|\big[\Phi_{1}(t_{1}\sqrt{L_{1}})\otimes \Phi_{2}(t_{2}\sqrt{L_{2}})
\big]^{\ast}_{\lambda_{1},\lambda_{2}}f(x_{1}, x_{2})\right|^{2} \frac{dt_{1}}{t_{1}}\frac{dt_{2}}{t_{2}},
\end{align*}
which readily yields the desired estimate.
\end{proof}

\begin{lemma} \label{211}
Suppose $\Phi_{1},\Phi_{2},\widetilde{\Phi}_{1}, \widetilde{\Phi}_{2}  \in \mathcal{A}(\mathbb{R})$ are even functions satisfying
\begin{equation*}
\Phi_{1}(0) = \Phi_{2}(0)= \widetilde{\Phi}_{1}(0) = \widetilde{\Phi}_{2}(0) =0.
\end{equation*}
Let $p \in (0, \infty)$, $\lambda_{1},\lambda_{2} >0$, and $w$ be arbitrary weight (i.e., non-negative locally integrable function)
on $X_{1} \times X_{2}$.  Then there exists a constant $C$ such that for all $f \in L^{2}(X_{1}\times X_{2})$,
\begin{equation} \label{poi}
\begin{split}
&\left\| \left(\int_{0}^{\infty}\int_{0}^{\infty}\left|\big[\widetilde{\Phi}_{1}(t_{1}\sqrt{L_{1}})
\otimes \widetilde{\Phi}_{2}(t_{2}\sqrt{L_{2}})\big]^{\ast}_{\lambda_{1},\lambda_{2}}f  \right|^{2}
\frac{dt_{1}}{t_{1}} \frac{dt_{2}}{t_{2}} \right)^{1/2}
\right\|_{L_{w}^{p}(X_{1}\times X_{2})} \\
&\quad \sim \left\| \left(\int_{0}^{\infty}\int_{0}^{\infty}\left|\big[\Phi_{1}(t_{1}\sqrt{L_{1}})\otimes \Phi_{2}(t_{2}\sqrt{L_{2}})
\big]^{\ast}_{\lambda_{1},\lambda_{2}}f  \right|^{2}  \frac{dt_{1}}{t_{1}} \frac{dt_{2}}{t_{2}} \right)^{1/2}
\right\|_{L_{w}^{p}(X_{1}\times X_{2})}.
\end{split}
\end{equation}
\end{lemma}
\begin{proof}
For $i=1,2$, since $\Phi_{i} \in \mathcal{A}(\mathbb{R})$ and $\Phi_{i}$ is even,
by Lemma \ref{uuuuu1} there exists an even function $\Theta_{i} \in \mathcal{S}(\mathbb{R})$
such that $\supp \Theta_{i} \subset \{\varepsilon_{i}/2 \leq |\lambda|\leq 2\varepsilon_{i}\}$ and
\begin{equation} \label{hpw12}
 \sum_{k =-\infty }^{\infty}\Phi_{i}(2^{-k}{\lambda})
\Theta_{i}(2^{-k}{\lambda})=1 \quad \mbox{for } \lambda \in \mathbb{R} \backslash \{0\},
\end{equation}
where $\varepsilon_{i}$ is the constant in the Tauberian condition \eqref{tauberian}
corresponding to $\Phi_{i}$. Hence it follows from Lemma \ref{negle} and the spectral theorem that
for all $f \in L^{2}(X_{1} \times X_{2})$ and $t_{1},t_{2} \in [1,2]$,
\begin{align*}
f = \sum_{k_{1}=-\infty}^{\infty} \sum_{k_{2}=-\infty}^{\infty}
\big(\Phi_{1}(2^{-k_{1}}t_{1}\sqrt{L_{1}})\Theta_{1}(2^{-k_{1}}t_{1}\sqrt{L_{1}})\big)
\otimes \big(\Phi_{2}(2^{- k_{2}}t_{2}\sqrt{L_{2}})\Theta_{2}(2^{-k_{2}}t_{2}\sqrt{L_{2}})\big)f
\end{align*}
with convergence in the sense of $L^{2}(X_{1} \times X_{2})$ norm.
Consequently, for all $j_{1}, j_{2} \in \mathbb{Z}$, all $t_{1},t_{2} \in [1,2]$ and a.e. $(y_{1}, y_{2}) \in X_{1} \times X_{2}$,
\begin{align} \label{sph2}
&\widetilde{\Phi}_{1}(2^{-j_{1}} t_{1}\sqrt{L_{1}}) \otimes \widetilde{\Phi}_{2}(2^{-j_{2}} t_{2}\sqrt{L_{2}}) f(y_{1},y_{2}) \\
 &=  \sum_{k_{1}=-\infty}^{\infty} \sum_{k_{2}=-\infty}^{\infty}
\big(\widetilde{\Phi}_{1}(2^{-j_{1}} t_{1}\sqrt{L_{1}})\Phi_{1}(2^{-k_{1}}t_{1}
\sqrt{L_{1}})\Theta_{1}(2^{-k_{1}}t_{1}\sqrt{L_{1}}) \big) \nonumber \\
&\quad\quad\quad\quad\quad\quad\quad\quad   \otimes \big(\widetilde{\Phi}_{2}(2^{-j_{2}} t_{2}
\sqrt{L_{2}})\Phi_{2}(2^{-k_{2}}t_{2}\sqrt{L_{2}})\Theta_{2}
(2^{-k_{2}}t_{2}\sqrt{L_{2}}) \big)f (y_{1},y_{2}) \nonumber\\
&  =  \sum_{k_{1}=-\infty}^{\infty} \sum_{k_{2}=-\infty}^{\infty}
\iint_{X_{1}\times X_{2}} K_{ \widetilde{\Phi}_{1}(2^{-j_{1}} t_{1}\sqrt{L_{1}})\Theta_{1}(2^{-k_{1}}t_{1}\sqrt{L_{1}})}(y_{1},z_{1})\nonumber \\
& \quad\quad\quad\quad\quad\quad\quad\quad\quad\quad\quad
\times K_{ \widetilde{\Phi}_{2}(2^{-j_{2}} t_{2}\sqrt{L_{2}})\Theta_{2}(2^{-k_{2}}t_{2}\sqrt{L_{2}}) }(y_{2},z_{2}) \nonumber\\
&\quad \quad \quad\quad\quad  \times \big(\Phi_{1}(2^{-k_{1}} t_{1}\sqrt{L_{1}})  \otimes \Phi_{2}(2^{-k_{2}} t_{2}\sqrt{L_{2}})\big) f(z_{1},z_{2})d\mu_{1}(z_{1})d\mu_{2}(z_{2}). \nonumber
\end{align}
Since $\widetilde{\Phi}_{i}$ is an even function on $\mathbb{R}$, we have $\widetilde{\Phi}_{i}'(0) =0$, for $i=1,2$.
Thus $\widetilde{\Phi}_{i} (0)= \widetilde{\Phi}_{i}'(0) =0$ for $i=1,2$.
On the other hand, since $\Theta_{i}$ vanishes near the origin, we have $\Theta_{i}^{(\nu)}(0)=0$ for every non-negative integer $\nu$.
Hence it follows from Lemma \ref{ODE} that for any positive integer $m$ and any $N >0$,
\begin{equation} \label{sph3}
\begin{split}
&\left|K_{\widetilde{\Phi}_{i}(2^{-j_{i}}t_{i}\sqrt{L_{i}})\Theta_{i}(2^{-k_{i}}t_{i}\sqrt{L_{2}})}(y_{2},z_{2})\right| \\
&\leq \begin{cases}
 C(\widetilde{\Phi}_{i}, \Theta_{i}, N)
 2^{-2|j_{i}-k_{i}|} V(y_{i},2^{-k_{i}}t_{i})^{-1}(1+2^{k_{i}}t_{i}^{-1}\rho_{i}(y_{i},z_{i}))^{-N}, &j_{i} \geq k_{i}, \\
C(\widetilde{\Phi}_{i}, \Theta_{i}, N, m)2^{-m|j_{i}-k_{i}|} V(y_{i},2^{-j_{i}}t_{i})^{-1}(1+2^{j_{i}}t_{i}^{-1}\rho_{i}(y_{i},z_{i}))^{-N},
 &j_{i} < k_{i}.
\end{cases}
\end{split}
\end{equation}
Choose $N \geq \max\{\lambda_{1} +  n_{1} + 1, \lambda_{2} + n_{2} +1\}$, then from \eqref{sph2}, \eqref{sph3} and
the inequality
\begin{align*}
&\left|\big(\Phi_{1}(2^{-k_{1}} t_{1}\sqrt{L_{1}})  \otimes \Phi_{2}(2^{-k_{2}} t_{2}\sqrt{L_{2}})\big) f(z_{1},z_{2})\right| \\
&\leq \big[\Phi_{1}(2^{-k_{1}}t_{1}\sqrt{L_{1}})\otimes \Phi_{2}(2^{-k_{2}}t_{2}\sqrt{L_{2}})
\big]^{\ast}_{\lambda_{1},\lambda_{2}}f (x_{1},x_{2}) \\
&\quad \times (1+2^{k_{1}}t_{1}^{-1}\rho_{1}(x_{1}, z_{1}))^{\lambda_{1}}
(1+2^{k_{2}}t_{2}^{-1}\rho_{2}(x_{2}, z_{2}))^{\lambda_{2}},
\end{align*}
we infer that
\begin{align*}
&\big[\widetilde{\Phi}_{1}(2^{-j_{1}}t_{1}\sqrt{L_{1}})\otimes \widetilde{\Phi}_{2}(2^{-j_{2}}t_{2}\sqrt{L_{2}})
\big]^{\ast}_{\lambda_{1},\lambda_{2}}f (x_{1},x_{2}) \\
& \leq \sum_{k_{1}=-\infty}^{\infty} \sum_{k_{2}=-\infty}^{\infty} \gamma_{j_{1}, k_{1}, j_{2}, k_{2}}
\big[\Phi_{1}(2^{-k_{1}}t_{1}\sqrt{L_{1}})\otimes \Phi_{2}(2^{-k_{2}}t_{2}\sqrt{L_{2}})
\big]^{\ast}_{\lambda_{1},\lambda_{2}}f (x_{1},x_{2}) \\
& \quad \times \esssup_{(y_{1},y_{2}) \in X_{1} \times X_{2}} \iint_{X_{1} \times X_{2}} \frac{  \prod_{i=1}^{2}(1+2^{k_{i}}t_{i}^{-1}\rho_{i}(x_{i}, z_{i}))^{\lambda_{i}}d\mu_{1}(z_{1})d\mu_{2}(z_{2})}{  \prod_{i=1}^{2}(1+2^{j_{i}}t_{i}^{-1}\rho_{i}(x_{i}, y_{i}))^{\lambda_{i}}  (1+2^{j_{i} \wedge k_{i}}t_{i}^{-1}\rho_{i}(y_{i}, z_{i}))^{\lambda_{i} + n_{i} +1} V(y_{i}, 2^{-(j_{i} \wedge k_{i})}t_{i})},
\end{align*}
where $j_{i} \wedge k_{i}:= \min\{j_{i}, k_{i}\}$ and
\begin{align*}
\gamma_{j_{1},k_{1},j_{2}, k_{2}}:=\begin{cases}
2^{-2|j_{1}-k_{1}|} 2^{-2|j_{2} -k_{2}|}  \quad & \mbox{if } j_{1} \geq k_{1} \mbox{ and } j_{2} \geq k_{2}, \\
2^{-2|j_{1}-k_{1}|} 2^{-m|j_{2} -k_{2}|} & \mbox{if } j_{1} \geq k_{1} \mbox{ and } j_{2} < k_{2}, \\
2^{-m|j_{1}-k_{1}|} 2^{-2|j_{2} -k_{2}|} & \mbox{if } j_{1} < k_{1} \mbox{ and } j_{2} \geq k_{2}, \\
2^{-m|j_{1}-k_{1}|} 2^{-m|j_{2} -k_{2}|} & \mbox{if } j_{1} \geq k_{1} \mbox{ and } j_{2} \geq k_{2}.
\end{cases}
\end{align*}
Using \eqref{inte} and the fundamental inequality
\begin{equation*}
(1+2^{k_{i}}t_{i}^{-1}\rho_{i}(x_{i}, z_{i}))^{\lambda_{i}}
\leq  \begin{cases}
 (1+2^{j_{i}}t_{i}^{-1}\rho_{i}(x_{i}, y_{i}))^{\lambda_{i}}  (1+2^{k_{i}}t_{i}^{-1}\rho_{i}(y_{i}, z_{i}))^{\lambda_{i}}, & j_{i}\geq k_{i},\\
2^{(k_{i}-j_{i})\lambda_{i}}(1+2^{j_{i}}t_{i}^{-1}\rho_{i}(x_{i}, y_{i}))^{\lambda_{i}}
(1+2^{j_{i}}t_{i}^{-1}\rho_{i}(y_{i}, z_{i}))^{\lambda_{i}}, \quad & j_{i} < k_{i},
\end{cases}
\end{equation*}
it follows that
\begin{equation} \label{wlg}
\begin{split}
&\big[\widetilde{\Phi}_{1}(2^{-j_{1}}t_{1}\sqrt{L_{1}})\otimes \widetilde{\Phi}_{2}(2^{-j_{2}}t_{2}\sqrt{L_{2}})
\big]^{\ast}_{\lambda_{1},\lambda_{2}}f (x_{1},x_{2}) \\
& \leq \sum_{k_{1}=-\infty}^{\infty} \sum_{k_{2}=-\infty}^{\infty} \gamma_{j_{1},k_{1},j_{2},k_{2}}'
\big[\Phi_{1}(2^{-k_{1}}t_{1}\sqrt{L_{1}})\otimes \Phi_{2}(2^{-k_{2}}t_{2}\sqrt{L_{2}})
\big]^{\ast}_{\lambda_{1},\lambda_{2}}f (x_{1},x_{2}),
\end{split}
\end{equation}
where
\begin{align*}
\gamma_{j_{1},k_{1},j_{2}, k_{2}}'
:=\begin{cases}
2^{-2|j_{1}-k_{1}|} 2^{-2|j_{2} -k_{2}|}  \quad & \mbox{if } j_{1} \geq k_{1} \mbox{ and } j_{2} \geq k_{2}, \\
2^{-2|j_{1}-k_{1}|} 2^{-(m-\lambda_{2})|j_{2} -k_{2}|} & \mbox{if } j_{1} \geq k_{1} \mbox{ and } j_{2} < k_{2}, \\
2^{-(m-\lambda_{1})|j_{1}-k_{1}|} 2^{-2|j_{2} -k_{2}|} & \mbox{if } j_{1} < k_{1} \mbox{ and } j_{2} \geq k_{2}, \\
2^{-(m-\lambda_{1})|j_{1}-k_{1}|} 2^{-(m-\lambda_{2})|j_{2} -k_{2}|} & \mbox{if } j_{1} \geq k_{1} \mbox{ and } j_{2} \geq k_{2}.
\end{cases}
\end{align*}
Now let us choose  $m > \max\{\lambda_{1}, \lambda_{2}\}$ and set $\sigma:= \min\{m- \lambda_{1}, m -\lambda_{2}, 2\}$.
Then \eqref{wlg} implies that
\begin{align*}
&\big[\widetilde{\Phi}_{1}(2^{-j_{1}}t_{1}\sqrt{L_{1}})\otimes \widetilde{\Phi}_{2}(2^{-j_{2}}t_{2}\sqrt{L_{2}})
\big]^{\ast}_{\lambda_{1},\lambda_{2}}f (x_{1},x_{2}) \\
& \leq   \sum_{k_{1}=-\infty}^{\infty} \sum_{k_{2}=-\infty}^{\infty} 2^{-|j_{1}-k_{1}|\sigma}2^{-|j_{2}-k_{2}|\sigma}
\big[\widetilde{\Phi}_{1}(2^{-k_{1}}t_{1}\sqrt{L_{1}})\otimes \widetilde{\Phi}_{2}(2^{-k_{2}}t_{2}\sqrt{L_{2}})
\big]^{\ast}_{\lambda_{1},\lambda_{2}}f (x_{1},x_{2}).
\end{align*}
Taking on both sides the norm $\left(\int_{1}^{2}\int_{1}^{2} |\cdot|^{2}\frac{dt_{1}}{t_{1}}\frac{dt_{2}}{t_{2}}\right)^{1/2}$ and
using Minkowski's inequality, we get
\begin{align*}
&\left(\int_{1}^{2}\int_{1}^{2}
\left\{\big[\widetilde{\Phi}_{1}(2^{-j_{1}}t_{1}\sqrt{L_{1}})\otimes \widetilde{\Phi}_{2}(2^{-j_{2}}t_{2}
\sqrt{L_{2}})\big]^{\ast}_{\lambda_{1},\lambda_{2}}f (x_{1},x_{2})  \right\}^{2}\frac{dt_{1}}{t_{1}}\frac{dt_{2}}{t_{2}} \right)^{1/2}\\
&\leq C\sum_{k_{1}=-\infty}^{\infty} \sum_{k_{2}=-\infty}^{\infty}
2^{-|j_{1}-k_{1}|\sigma}2^{-|j_{2}-k_{2}|\sigma}\\
&\quad\quad \times\left(\int_{1}^{2}\int_{1}^{2} \left\{\big[\Phi_{1}
(2^{-k_{1}}t_{1}\sqrt{L_{1}})\otimes \Phi_{2}(2^{-j_{2}}t_{2}\sqrt{L_{2}})
\big]^{\ast}_{\lambda_{1},\lambda_{2}}f (x_{1},x_{2}) \right\}^{2}\frac{dt_{1}}{t_{1}}\frac{dt_{2}}{t_{2}} \right)^{1/2}.
\end{align*}
Finally, applying Lemma \ref{RY} in $L^{p}(\ell^{2})$ yields
\begin{equation} \label{70692}
\begin{split}
&\left\| \left(\int_{0}^{\infty}\int_{0}^{\infty}\left|\big[\widetilde{\Phi}_{1}(t_{1}\sqrt{L_{1}})
\otimes \widetilde{\Phi}_{2}(t_{2}\sqrt{L_{2}})\big]^{\ast}_{\lambda_{1},\lambda_{2}}f  \right|^{2}
\frac{dt_{1}}{t_{1}} \frac{dt_{2}}{t_{2}} \right)^{1/2}
\right\|_{L_{w}^{p}(X_{1}\times X_{2})} \\
&\quad \leq C \left\| \left(\int_{0}^{\infty}\int_{0}^{\infty}\left|\big[\Phi_{1}(t_{1}\sqrt{L_{1}})\otimes \Phi_{2}(t_{2}\sqrt{L_{2}})
\big]^{\ast}_{\lambda_{1},\lambda_{2}}f  \right|^{2}  \frac{dt_{1}}{t_{1}} \frac{dt_{2}}{t_{2}} \right)^{1/2}
\right\|_{L_{w}^{p}(X_{1}\times X_{2})}.
\end{split}
\end{equation}
By symmetry, the converse inequality of \eqref{70692} also holds. The proof of the lemma is complete.
\end{proof}

\begin{lemma} \label{central}
Let $\Phi_{1},\Phi_{2} \in \mathcal{A}(\mathbb{R})$ be even functions.
Then for any $r>0$, $\sigma  >0$, $\lambda_{1} >D_{1}/2$ and $\lambda_{2} > D_{2}/2$,
there exists a constant $C$ such that for all $f \in L^{2}(X_{1} \times X_{2})$,
all $(x_{1},x_{2}) \in X_{1} \times X_{2}$ and all $t_{1},t_{2} \in [1,2]$,
\begin{align} \label{cen}
&\left\{\big[\Phi_{1}(2^{-j_{1}}t_{1}\sqrt{L_{1}})\otimes \Phi_{2}(2^{-j_{2}}t_{2}\sqrt{L_{2}})
\big]^{\ast}_{\lambda_{1},\lambda_{2}}f(x_{1}, x_{2}) \right\}^{r} \\
&\leq C \sum_{k_{1}=j_{1} }^{\infty}\sum_{k_{2}=j_{2}}^{\infty}2^{(j_{1}-k_{1})\sigma}2^{(j_{2}-k_{2})\sigma}  \nonumber\\
& \quad \times \iint_{X_{1}\times X_{2}}
\frac{\left|\Phi_{1}(2^{-k_{1}}t_{1}\sqrt{L_{1}})\otimes \Phi_{2}(2^{-k_{2}}t_{2}\sqrt{L_{2}})
f(z_{1},z_{2})\right|^{r}d\mu_{1}(z_{1})d\mu_{2}(z_{2})}{V(z_{1},2^{-k_{1}}t_{1})(1 + 2^{k_{1}}t_{1}^{-1}\rho(x_{1},z_{1}))^{\lambda_{1} r}
V(z_{2},2^{-k_{2}}t_{2})(1 + 2^{k_{2}}t_{2}^{-1}\rho_{2}(x_{2},z_{2}))^{\lambda_{2} r}}. \nonumber
\end{align}
\end{lemma}

\begin{proof}
By Lemma \ref{uuuuu}, for $i=1,2$
there exist even functions $\Psi_{i}, \Upsilon_{i}, \Theta_{i} \in \mathcal{S}(\mathbb{R})$
such that $\supp \Upsilon_{i} \subset \{|\lambda| \leq 2\varepsilon_{i}\}$,
$\supp \Theta_{i} \subset \{\varepsilon_{i}/2 \leq |\lambda|\leq 2\varepsilon_{i}\}$, and
\begin{equation} \label{hpw1}
\Psi_{i}(\lambda)\Upsilon_{i}(\lambda)+ \sum_{k_{i}=1}^{\infty}\Phi_{i}(2^{-k_{i}}\lambda)
\Theta_{i}(2^{-k_{i}}\lambda)=1 \quad \mbox{for all } \lambda \in \mathbb{R},
\end{equation}
where $\varepsilon_{i}$ is the constant in the Tauberian condition \eqref{tauberian}
corresponding to $\Phi_{i}$. Replacing $\lambda$ with $2^{-j_{i}}t_{i} \lambda$ in \eqref{hpw1},
we see that for all $j_{i} \in \mathbb{Z}$ and $t_{i}\in [1,2]$,
\begin{equation*}
\Psi_{i}(2^{-j_{i}}t_{i}\lambda)\Upsilon_{i}(2^{-j_{i}}t_{i}\lambda) +
\sum_{k_{i}=1}^{\infty}\Phi_{i}(2^{-(k_{i}+ j_{i})}t_{i}\lambda)\Theta_{i}(2^{-(k_{i}+j_{i})}t_{i}\lambda)=1.
\end{equation*}
It then follows from the spectral theorem that for all $f \in L^{2}(X_{1} \times X_{2})$, all $j_{1}, j_{2} \in \mathbb{Z}$
and all $t_{1},t_{2} \in [1,2]$,
\begin{align*}
f &= \big(\Psi_{1}(2^{-j_{1}}t_{1}\sqrt{L_{1}})\Upsilon_{1}(2^{-j_{1}}t_{1}\sqrt{L_{1}}) \big)
\otimes   \big(\Psi_{2}(2^{-j_{2}}t_{2}\sqrt{L_{2}})\Upsilon_{2}(2^{-j_{2}}t_{2}\sqrt{L_{2}} )\big)f \\
&\quad  + \sum_{k_{1}=1}^{\infty}\big(\Phi_{1}(2^{-(k_{1}+j_{1})}t_{1}\sqrt{L_{1}})\Theta_{1}(2^{-(k_{1}+j_{1})}t_{1}\sqrt{L_{1}}) \big)
\otimes  \big(\Psi_{2}(2^{-j_{2}}t_{2}\sqrt{L_{2}})\Upsilon_{2}(2^{-j_{2}}t_{2}\sqrt{L_{2}} )\big)f \\
&\quad + \sum_{k_{2}=1}^{\infty} \big(\Psi_{1}(2^{-j_{1}}t_{1}\sqrt{L_{1}})\Upsilon_{1}(2^{-j_{1}}t_{1}\sqrt{L_{1}}) \big)
\otimes   \big(\Phi_{2}(2^{- (k_{2}+ j_{2})}t_{2}\sqrt{L_{2}})\Theta_{2}(2^{-(k_{2} + j_{2})}t_{2}\sqrt{L_{2}}) \big)f\\
&\quad + \sum_{k_{1}=1}^{\infty} \sum_{k_{2}=1}^{\infty}
\big(\Phi_{1}(2^{-(k_{1}+j_{1})}t_{1}\sqrt{L_{1}})\Theta_{1}(2^{-(k_{1}+j_{1})}t_{1}\sqrt{L_{1}}) \big)
\otimes \big(\Phi_{2}(2^{- (k_{2}+ j_{2})}t_{2}\sqrt{L_{2}})\Theta_{2}(2^{-(k_{2} + j_{2})}t_{2}\sqrt{L_{2}}) \big)f
\end{align*}
with convergence in the sense of $L^{2}(X_{1} \times X_{2})$ norm. Hence, for all $j_{1}, j_{2}\in \mathbb{Z}$
and  a.e. $(y_{1},y_{2})\in X_{1}\times X_{2}$, we have
\begin{align} \label{sph}
&  \Phi_{1}(2^{-j_{1}} t_{1}\sqrt{L_{1}})
\otimes \Phi_{2}(2^{-j_{2}} t_{2}\sqrt{L_{2}}) f(y_{1},y_{2}) \\
&=\big(\Phi_{1}(2^{-j_{1}} t_{1} \sqrt{L_{1}}) \Psi_{1}(2^{-j_{1}}t_{1}\sqrt{L_{1}})
\Upsilon_{1}(2^{-j_{1}}t_{1}\sqrt{L_{1}}) \big)\nonumber\\
&\quad\quad\quad\quad\quad\quad
\otimes \big(\Phi_{2}(2^{-j_{2}} t_{2}\sqrt{L_{2}})\Psi_{2}(2^{-j_{2}}t_{2}\sqrt{L_{2}})
\Upsilon_{2}(2^{-j_{2}}t_{2}\sqrt{L_{2}})\big)f(y_{1},y_{2}) \nonumber\\
&\quad  + \sum_{k_{1}=1}^{\infty}\big(\Phi_{1}(2^{-j_{1}} t_{1} \sqrt{L_{1}}) \Phi_{1}(2^{-(k_{1}+j_{1})}t_{1}\sqrt{L_{1}})\Theta_{1}(2^{-(k_{1}+j_{1})}t_{1}\sqrt{L_{1}}) \big)\nonumber\\
&\quad\quad\quad\quad\quad \otimes  \big(\Phi_{2}(2^{-j_{2}} t_{2}\sqrt{L_{2}})\Psi_{2}(2^{-j_{2}}t_{2}\sqrt{L_{2}})\Upsilon_{2}(2^{-j_{2}}t_{2}\sqrt{L_{2}} )\big)f(y_{1},y_{2}) \nonumber\\
&\quad + \sum_{k_{2}=1}^{\infty} \big(\Phi_{1}(2^{-j_{1}} t_{1} \sqrt{L_{1}}) \Psi_{1}(2^{-j_{1}}t_{1}\sqrt{L_{1}})\Upsilon_{1}(2^{-j_{1}}t_{1}\sqrt{L_{1}}) \big)\nonumber\\
& \quad\quad\quad\quad\quad \otimes   \big(\Phi_{2}(2^{-j_{2}} t_{2}\sqrt{L_{2}})\Phi_{2}(2^{- (k_{2}+ j_{2})}t_{2}\sqrt{L_{2}})\Theta_{2}(2^{-(k_{2} + j_{2})}t_{2}\sqrt{L_{2}}) \big)f\nonumber\\
&\quad + \sum_{k_{1}=1}^{\infty} \sum_{k_{2}=1}^{\infty}
\big(\Phi_{1}(2^{-j_{1}} t_{1}\sqrt{L_{1}})\Phi_{1}(2^{-(k_{1}+ j_{1})}t_{1}\sqrt{L_{1}})
\Theta_{1}(2^{-(k_{1}+j_{1})}t_{1}\sqrt{L_{1}}) \big)\nonumber\\
&\quad\quad\quad\quad\quad\quad
 \otimes \big(\Phi_{2}(2^{-j_{2}} t_{2}\sqrt{L_{2}})\Phi_{2}(2^{-(k_{2}+j_{2})}t_{2}\sqrt{L_{2}})
 \Theta_{2}(2^{-(k_{2} +j_{2})}t_{2}\sqrt{L_{2}}) \big)f (y_{1},y_{2})\nonumber\\
&  = \iint_{X_{1}\times X_{2}} K_{\Psi_{1}(2^{-j_{1}}t_{1}\sqrt{L_{1}})\Upsilon_{1}(2^{-j_{1}}t_{1}\sqrt{L_{1}}) }(y_{1},z_{1})
 K_{\Psi_{2}(2^{-j_{2}}t_{2}\sqrt{L_{2}})\Upsilon_{2}(2^{-j_{2}}t_{2}\sqrt{L_{2}})}(y_{2},z_{2}) \nonumber\\
&\quad\quad\quad\quad\quad\quad
 \times \big(\Phi_{1}(2^{-(0+j_{1})} t_{1}\sqrt{L_{1}}) \otimes \Phi_{2}(2^{-(0+j_{2})} t_{2}\sqrt{L_{1}})
 \big)f(z_{1},z_{2})d\mu(z_{1})d\mu(z)\nonumber\\
 & \quad  + \sum_{k_{1}=1}^{\infty}\iint_{X_{1}\times X_{2}} K_{\Phi_{1}(2^{-j_{1}}t_{1}\sqrt{L_{1}})\Theta_{1}(2^{-(k_{1}+j_{1})}t_{1}\sqrt{L_{1}}) }(y_{1},z_{1})
 K_{\Psi_{2}(2^{-j_{2}}t_{2}\sqrt{L_{2}})\Upsilon_{2}(2^{-j_{2}}t_{2}\sqrt{L_{2}})}(y_{2},z_{2})\nonumber \\
&\quad\quad\quad\quad\quad\quad
 \times \big(\Phi_{1}(2^{-(k_{1}+j_{1})} t_{1}\sqrt{L_{1}}) \otimes \Phi_{2}(2^{-(0+j_{2})} t_{2}\sqrt{L_{1}})
 \big)f(z_{1},z_{2})d\mu(z_{1})d\mu(z)\nonumber\\
  & \quad  + \sum_{k_{2}=1}^{\infty}\iint_{X_{1}\times X_{2}} K_{\Psi_{1}(2^{-j_{1}}t_{1}\sqrt{L_{1}})\Upsilon_{1}(2^{-j_{1}}t_{1}\sqrt{L_{1}}) }(y_{1},z_{1})
 K_{\Phi_{2}(2^{-j_{2}}t_{2}\sqrt{L_{2}})\Theta_{2}(2^{-(k_{2}+j_{2})}t_{2}\sqrt{L_{2}})}(y_{2},z_{2}) \nonumber\\
&\quad\quad\quad\quad\quad\quad
 \times \big(\Phi_{1}(2^{-(0+j_{1})} t_{1}\sqrt{L_{1}}) \otimes \Phi_{2}(2^{-(k_{2}+j_{2})} t_{2}\sqrt{L_{1}})
 \big)f(z_{1},z_{2})d\mu(z_{1})d\mu(z)\nonumber\\
&\quad + \sum_{k_{1}=1}^{\infty} \sum_{k_{2}=1}^{\infty}
\iint_{X_{1}\times X_{2}} K_{ \Phi_{1}(2^{-j_{1}} t_{1}\sqrt{L_{1}})\Theta_{1}(2^{-(k_{1}+j_{1})}t_{1}\sqrt{L_{1}})}(y_{1},z_{1})\nonumber\\
&\quad\quad\quad\quad\quad\quad\quad\quad\quad \times
 K_{ \Phi_{2}(2^{-j_{2}} t_{2}\sqrt{L_{2}})\Theta_{2}(2^{-(k_{2} + j_{2})}t_{2}\sqrt{L_{2}}) }(y_{2},z_{2})\nonumber \\
&\quad\quad\quad\quad\quad\quad\quad\quad\quad
 \times \big(\Phi_{1}(2^{-(k_{1}+j_{1})} t_{1}\sqrt{L_{1}})
\otimes \Phi_{2}(2^{-(k_{2}+j_{2})} t_{2}\sqrt{L_{2}}) \big)f(z_{1},z_{2})d\mu(z_{1})d\mu(z).\nonumber
\end{align}

For $i=1,2$, let $N_{i} \geq \lambda_{i}$ and $m_{i}$ be any integer such that $m_{i} - \lambda_{i} -n_{i}/r >0$.
Since $\Theta_{i}$ vanishes near the origin, it follows from Lemma \ref{ODE} that
there exists a constant $C=C(\Phi_{i},\Theta_{i}, m_{i}, N_{i})$
such that for all $j_{i} \in \mathbb{Z}$, all $k_{i} \in \{1,2, \cdots\}$, and all $t_{i} \in [1,2]$,
\begin{equation} \label{ker1}
\big|K_{ \Phi_{i}(2^{-j_{i}} t_{i}\sqrt{L_{1}})\Theta_{i}(2^{-(k_{i}+j_{i})}t_{i}\sqrt{L_{i}})  }(y_{i},z_{i})\big|
\leq C  2^{-k_{i}m_{i}}V(z_{i},2^{-j_{i}}t_{i})^{-1}(1+2^{j_{i}}t_{i}^{-1}\rho_{i}(y_{i},z_{i}))^{-N_{i}}.
\end{equation}
Analogously, for $i=1,2$, we have
\begin{equation} \label{ker2}
\big| K_{\Psi_{i}(2^{-j_{i}}t_{i}\sqrt{L_{i}})\Upsilon_{i}(2^{-j_{i}}t_{i}\sqrt{L_{i}})}(y_{i},z_{i})\big|
\leq CV(z_{i},2^{-j_{i}}t_{i})^{-1}(1+2^{j_{i}}t_{i}^{-1}\rho_{i}(y_{i},z_{i}))^{-N_{i}}.
\end{equation}
Putting \eqref{ker1} and \eqref{ker2} into \eqref{sph}, we obtain
\begin{align}\label{ftqkf}
&\left|\Phi_{1}(2^{-j_{1}} t_{1}\sqrt{L_{1}}) \otimes \Phi_{2}(2^{-j_{2}} t_{2}\sqrt{L_{2}}) f(y_{1},y_{2})\right|\\
&\leq C\sum_{k_{1}=0}^{\infty}\sum_{k_{2}=0}^{\infty} 2^{-k_{1}m_{1} }2^{-k_{2}m_{2}} \nonumber \\
&\quad\quad \times \iint_{X_{1} \times X_{2}}  \frac{\left|\Phi_{1}(2^{-(k_{1}+j_{1})} t_{1}\sqrt{L_{1}})
\otimes \Phi_{2}(2^{-(k_{2}+j_{2})} t_{2}\sqrt{L_{2}})f(z_{1},z_{2})\right|}
{\prod_{i=1}^{2}V(z_{i},2^{-j_{i}}t_{i})(1+2^{j_{i}}t_{i}^{-1}\rho_{i}(y_{i},z_{i}))^{N_{i}}}
 d\mu_{1}(z_{1})d\mu_{2}(z_{2}) \nonumber \\
& = C\sum_{k_{1}=j_{1}}^{\infty}\sum_{k_{2}=j_{2}}^{\infty} 2^{(j_{1}-k_{1})m_{1} }2^{(j_{2}-k_{2})m_{2}}\nonumber  \\
&\quad\quad \times
\iint_{X_{1} \times X_{2}}\frac{\left|\Phi_{1}(2^{-k_{1}}t_{1}\sqrt{L_{1}})
\otimes \Phi_{2}(2^{-k_{2}} t_{2}\sqrt{L_{2}})f(z_{1},z_{2})\right|}
{\prod_{i=1}^{2}V(z_{i},2^{-j_{i}}t_{i})(1+2^{j_{i}}t_{i}^{-1}\rho_{i}(y_{i},z_{i}))^{N_{i}}}
d\mu_{1}(z_{1})d\mu_{2}(z_{2}).\nonumber
\end{align}

To prove the desired inequality, we first consider the case $0< r \leq 1$.
Dividing both sides of \eqref{ftqkf} by $(1 +2^{j_{1}}t_{1}^{-1}\rho_{1}(x_{1},y_{1}))^{\lambda_{1}}
(1 +2^{j_{2}}t_{2}^{-1}\rho_{2}(x_{2},y_{2}))^{\lambda_{2}}$,
taking the supremum over $(y_{1},y_{2}) \in X_{1} \times X_{2}$ in the left-hand side, and using the inequalities
$V(z,2^{-j_{i}}t_{i}) \geq V(z_{i},2^{-k_{i}}t_{i})$ ($\forall k_{i} \geq j_{i}$)
and $(1 +2^{j_{i}}t_{i}^{-1}\rho_{i}(x_{i},y_{i}))(1 +2^{j_{i}}t_{i}^{-1}
\rho_{i}(y_{i}, z_{i})) \geq (1 +2^{j_{i}}t_{i}^{-1}\rho_{i}(x_{i},z_{i}))$ ($\forall
t_{i} \in [1,2]$) in the right-hand side, we get that, for all $t_{i} \in [1,2]$ and $x_{i} \in X_{i}$,
\begin{align} \label{shibian}
&\big[\Phi_{1}(2^{-j_{1}}t_{1}\sqrt{L_{1}})\otimes \Phi_{2}(2^{-j_{2}}t_{2}\sqrt{L_{2}})
\big]^{\ast}_{\lambda_{1},\lambda_{2}}f(x_{1}, x_{2})\\
&\leq C\sum_{k_{1}=j_{1} }^{\infty}\sum_{k_{2} =j_{2}}^{\infty} 2^{(j_{1}-k_{1})m_{1}}
2^{(j_{2}-k_{2})m_{2}}  \nonumber\\
& \quad\quad \times \iint_{X_{1} \times X_{2}}   \frac{\left|\Phi_{1}(2^{-k_{1}}
 t_{1}\sqrt{L_{1}}) \otimes \Phi_{2}(2^{-k_{2}} t_{2}\sqrt{L_{2}})f(z_{1},z_{2})\right|}
{\prod_{i=1}^{2}V(z_{i},2^{-k_{i}}t_{i})(1+2^{j_{i}}t_{i}^{-1}\rho_{i}(x_{i},z_{i}))^{\lambda_{i}}}
d\mu_{1}(z_{1})d\mu_{2}(z_{2}). \nonumber
\end{align}
To proceed further, we note that
\begin{align} \label{hanshu}
& \big|\Phi_{1}(2^{-k_{1}}
 t_{1}\sqrt{L_{1}}) \otimes \Phi_{2}(2^{-k_{2}} t_{2}\sqrt{L_{2}}) f(z_{1},z_{2})\big|\\
& \leq \big|\Phi_{1}(2^{-k_{1}} t_{1}\sqrt{L_{1}}) \otimes \Phi_{2}(2^{-k_{2}} t_{2}
\sqrt{L_{2}}) f(z_{1},z_{2})\big|^{r}\nonumber \\
&\quad \times \left\{\big[\Phi_{1}(2^{-k_{1}}t_{1}\sqrt{L_{1}})\otimes \Phi_{2}(2^{-k_{2}}t_{2}
\sqrt{L_{2}})\big]^{\ast}_{\lambda_{1},\lambda_{2}}f(x_{1}, x_{2})\right\}^{1-r}\nonumber\\
&\quad \times (1 +2^{k_{1}}t_{1}^{-1}\rho_{1}(x_{1},z_{1}))^{\lambda_{1}(1-r)}(1 +2^{k_{2}}t_{2}^{-1}\rho_{2}
(x_{2},z_{2}))^{\lambda_{2}(1-r)}.\nonumber
\end{align}
From \eqref{shibian}, \eqref{hanshu}, and the inequality
\begin{align*}
(1 +2^{k_{i}}t_{i}^{-1}\rho_{i}(x_{i},z_{i}))^{\lambda_{i}}
\leq 2^{(k_{i}-j_{i})\lambda_{i}}(1+2^{j_{i}}t_{i}^{-1}\rho_{i}(x_{i},z_{i}))^{\lambda_{i}} \quad (\forall k_{i} \geq j_{i}, ~\forall t_{i} \in [1,2]),
\end{align*}
it follows that
\begin{align} \label{317}
&\big[\Phi_{1}(2^{-j_{1}}t_{1}\sqrt{L_{1}})\otimes \Phi_{2}(2^{-j_{2}}t_{2}\sqrt{L_{2}})
\big]^{\ast}_{\lambda_{1},\lambda_{2}}f(x_{1}, x_{2}) \\
& \leq C\sum_{k_{1}= j_{1} }^{\infty}\sum_{k_{2} = j_{2}}^{\infty} 2^{(j_{1}- k_{1})(m_{1}-\lambda_{1})}
2^{(j_{2}- k_{2})(m_{2}-\lambda_{2})}\nonumber \\
&\quad\quad \times \iint_{X_{1}\times X_{2}} \frac{\left|\Phi_{1}(2^{-k_{1}}
 t_{1} \sqrt{L_{1}}) \otimes \Phi_{2}(2^{-k_{2}} t_{2} \sqrt{L_{2}}) f(z_{1},z_{2})\right|^{r}}{\prod_{i=1}^{2}V(z_{i},2^{-k_{i}}t_{i})
 (1+2^{k_{i}}t_{i}^{-1}\rho_{i}(x_{i},z_{i}))^{\lambda_{i} r}}
 d\mu_{1}(z_{1})d\mu_{2}(z_{2}) \nonumber\\
& \quad\quad\times \left\{ \big[\Phi_{1}(2^{-k_{1}}t_{1}\sqrt{L_{1}})\otimes \Phi_{2}(2^{-k_{2}}t_{2}\sqrt{L_{2}})
\big]^{\ast}_{\lambda_{1},\lambda_{2}}f(x_{1}, x_{2})\right\}^{1-r}.\nonumber
\end{align}

We claim that for any $f \in L^{2}(X_{1} \times X_{2})$,
$\lambda_{i}>D_{i}/2$, $x_{i}\in X_{i}$, $t_{i}\in [1,2]$, and $j_{i} \in \mathbb{Z}$,
\begin{equation} \label{318}
\big[\Phi_{1}(2^{-j_{1}}t_{1}\sqrt{L_{1}})\otimes \Phi_{2}(2^{-j_{2}}t_{2}\sqrt{L_{2}})
\big]^{\ast}_{\lambda_{1},\lambda_{2}}f(x_{1}, x_{2})<\infty,
\end{equation}
and there exists $N_{0}>0$ such that
\begin{equation} \label{319}
\big[\Phi_{1}(2^{-j_{1}}t_{1}\sqrt{L_{1}})\otimes \Phi_{2}(2^{-j_{2}}t_{2}\sqrt{L_{2}})
\big]^{\ast}_{\lambda_{1},\lambda_{2}}f(x_{1}, x_{2}) =O(2^{j_{1}N_{0}}2^{j_{2}N_{0}})
\end{equation}
as $j_{1}, j_{2} \rightarrow +\infty$. Indeed, for $i=1,2$, by Lemma \ref{smooth} we have
\begin{align*}
\left|K_{\Phi_{i}(2^{-j_{i}}t_{i}\sqrt{L_{i}})}(y_{i},z_{i})\right|
&\leq CV(y_{i},2^{-j_{i}}t_{i})^{-1} (1+2^{j_{i}}t_{i}^{-1}\rho(y_{i},z_{i}))^{-(n_{i}+1)/2}.
\end{align*}
Hence, by the Cauchy-Schwartz inequality and \eqref{inte}, we have
\begin{align*}
&\left|\Phi_{1}(2^{-j_{1}}t_{1}\sqrt{L_{1}})\otimes \Phi_{2}(2^{-j_{2}}t_{2}\sqrt{L_{2}})
f(y_{1}, y_{2})\right|  \\
&\leq C \iint_{X_{1} \times X_{2}} \big| K_{\Phi_{1}(2^{-j_{1}}t_{1}\sqrt{L_{1}})}(y_{1},z_{1}) \big|
 \big| K_{\Phi_{2}(2^{-j_{2}}t_{2}\sqrt{L_{2}})}(y_{2},z_{2}) \big||f(z_{1},z_{2})|d\mu(z_{1})d\mu(z_{2}) \\
& \leq  C\|f\|_{L^{2}(X_{1}\times X_{2})}V(y_{1},2^{-j_{1}}t_{1})^{-1/2}V(y_{2},2^{-j_{2}}t_{1})^{-1/2}.
\end{align*}
This along with \eqref{tos} yields that for $\lambda_{i} \geq D_{i}/2$,
\begin{align*}
& \big[\Phi_{1}(2^{-j_{1}}t_{1}\sqrt{L_{1}})\otimes \Phi_{2}(2^{-j_{2}}t_{2}\sqrt{L_{2}})
\big]^{\ast}_{\lambda_{1},\lambda_{2}}f(x_{1}, x_{2}) \\
& \leq C\esssup_{(y_{1}, y_{2}) \in X_{1} \times X_{2}}\frac{\|f\|_{L^{2}(X_{1}\times X_{2})}}
{\prod_{i=1}^{2} V(y_{i},2^{-j_{i}}t_{i})^{-1/2}
(1+ 2^{j_{i}}t_{i}^{-1}\rho_{i}(x_{i},y_{i}))^{\lambda_{i}}}\\
& \leq C \|f\|_{L^{2}(X_{1}\times X_{2})} V(x_{1},2^{-j_{1}}t_{1})^{-1/2}V(x_{2},2^{-j_{2}}t_{2})^{-1/2}.
\end{align*}
Hence \eqref{318} is true. Moreover, if $j_{1}, j_{2} \geq 1$,
by \eqref{hom} we have
\begin{align*}
& \big[\Phi_{1}(2^{-j_{1}}t_{1}\sqrt{L_{1}})\otimes \Phi_{2}(2^{-j_{2}}t_{2}\sqrt{L_{2}})
\big]^{\ast}_{\lambda_{1},\lambda_{2}}f(x_{1}, x_{2}) \\
& \leq C \|f\|_{L^{2}(X_{1}\times X_{2})} V(x_{1},2^{-j_{1}}t_{1})^{-1/2}V(x_{2},2^{-j_{2}}t_{2})^{-1/2}\\
&  \leq C2^{j_{1}n_{1}/2}2^{j_{2}n_{2}/2}\|f\|_{L^{2}(X_{1}\times X_{2})} V(x_{1},1)^{-1/2}V(x_{2},1)^{-1/2},
\end{align*}
which verifies \eqref{319} with $N_{0} = \max\{n_{1}/2,n_{2}/2\}$.

Since $m_{1}, m_{2}$ in \eqref{317} can be chosen to be arbitrarily large, it follows
from \eqref{317}, \eqref{318}, \eqref{319} and Lemma \ref{lpm} that for any $\sigma >0$,
\begin{align*}
&\left\{\big[\Phi_{1}(2^{-j_{1}}t_{1}\sqrt{L_{1}})\otimes \Phi_{2}(2^{-j_{2}}t_{2}\sqrt{L_{2}})
\big]^{\ast}_{\lambda_{1},\lambda_{2}}f(x_{1}, x_{2})\right\}^{r}\\
& \leq C\sum_{k_{1}=j_{1}}^{\infty}\sum_{k_{2} =j_{2}}^{\infty} 2^{(j_{1}- k_{1})\sigma}
2^{(j_{2}- k_{2})\sigma}\\
&\quad\quad \times \iint_{X_{1}\times X_{2}}\frac{\left|\Phi_{1}(2^{-k_{1}}
 t_{1}\sqrt{L_{1}}) \otimes \Phi_{2}(2^{-k_{2}} t_{2}\sqrt{L_{2}})f(z_{1},z_{2})\right|^{r}  d\mu_{1}(z_{1})d\mu_{2}(z_{2})}
 {V(z_{1},2^{-k_{1}}t_{1})
 (1+2^{k_{1}}t_{1}^{-1}\rho_{1}(x_{1},z_{1}))^{\lambda_{1} r}V(z_{2},2^{-k_{2}}t_{2})(1+2^{k_{2}}t_{2}^{-1}
 \rho_{2}(x_{2},z_{2}))^{\lambda_{2} r}}.
\end{align*}
This proves \eqref{cen} for $0<r \leq  1$.

Next we show \eqref{cen} for $r>1$. Indeed, from \eqref{ftqkf} with
$m_{i} \geq \sigma + \lambda_{i}r + \varepsilon$ and $N_{i} \geq \lambda_{i} + (D_{i}+n_{i}+1)/r'$,
where $\varepsilon$ is any fixed positive number and $r'$ is a number such that $1/r+1/r'=1$, it follows that
\begin{align*}
&\left|\Phi_{1}(2^{-j_{1}} t_{1}\sqrt{L_{1}})\otimes \Phi_{2}(2^{-j_{2}} t_{2}\sqrt{L_{2}})f(y_{1},y_{2})\right|\\
& \leq C\sum_{k_{1}=j_{1}}^{\infty}\sum_{k_{2}=j_{2}}^{\infty} 2^{(j_{1}-k_{1})(\sigma + \lambda_{1}r+
\varepsilon)}2^{(j_{2}-k_{2})(\sigma + \lambda_{2}r + \varepsilon)}\\
&\quad\quad\quad\quad\quad\quad \times \iint_{X_{1} \times X_{2}}  \frac{\left|\Phi_{1}(2^{-k_{1}} t_{1}\sqrt{L_{1}})
\otimes \Phi_{2}(2^{-k_{2}} t_{2}\sqrt{L_{2}})f(z_{1},z_{2})\right|}
{\prod_{i=1}^{2}V(z_{i},2^{-j_{i}}t_{i})(1+2^{j_{i}}t_{i}^{-1}\rho_{i}(y_{i},z_{i}))^{\lambda_{i} + (D_{i}+n_{i}+1)/r'}
} d\mu_{1}(z_{1})d\mu_{2}(z_{2})\\
& \leq C\sum_{k_{1}=j_{1}}^{\infty}\sum_{k_{2}=j_{2}}^{\infty} 2^{(j_{1}-k_{1})
(\sigma +\lambda_{1}r+ \varepsilon)}2^{(j_{2}-k_{2})(\sigma + \lambda_{2}r+\varepsilon)}\\
&\quad\quad\quad\quad\quad\quad \times \left(\iint_{X_{1} \times X_{2}}  \frac{\left|\Phi_{1}(2^{-k_{1}} t_{1} \sqrt{L_{1}})
\otimes \Phi_{2}(2^{-k_{2}} t_{2}\sqrt{L_{2}})f(z_{1},z_{2})\right|^{r}}
{\prod_{i=1}^{2}V(z_{i},2^{-j_{i}}t_{i})(1+2^{j_{i}}t_{i}^{-1}\rho_{i}(y_{i},z_{i}))^{\lambda_{i}r}} d\mu_{1}(z_{1})d\mu_{2}(z_{2})\right)^{1/r}\\
& \leq C\left(\sum_{k_{1}=j_{1}}^{\infty}\sum_{k_{2}=j_{2}}^{\infty} 2^{(j_{1}-k_{1})(\sigma_{1}+\lambda_{1}r)}2^{(j_{2}-k_{2})(\sigma_{2}+\lambda_{2}r)}\right.\\
& \quad\quad\quad\quad\quad\quad
\left. \times \iint_{X_{1} \times X_{2}}  \frac{\left|\Phi_{1}(2^{-k_{1}} t_{1}\sqrt{L_{1}}) \otimes \Phi_{2}(2^{-k_{2}} t_{2}\sqrt{L_{2}}) f(z_{1},z_{2})\right|^{r}}
{\prod_{i=1}^{2}V(z_{i},2^{-j_{i}}t_{i})(1+2^{j_{i}}t_{i}^{-1}\rho_{i}(y_{i},z_{i}))^{\lambda_{i}r}
} d\mu_{1}(z_{1})d\mu_{2}(z_{2})\right)^{1/r},
\end{align*}
where we applied H\"{o}lder's inequality for the integrals and the sums, and used
\eqref{tos} and \eqref{inte}. Raising both sides to the power $r$, dividing both
sides by $(1+ 2^{j_{1}}t_{1}^{-1}\rho_{1}(x_{1},y_{1}))^{\lambda_{1} r}(1+ 2^{j_{2}}t_{2}^{-1}\rho_{2}(x_{2},y_{2}))^{\lambda_{2} r}$,
in the left-hand side taking the supremum over $(y_{1},y_{2}) \in X_{1} \times X_{2}$, and in the right-hand side using the inequalities
\begin{align*}
&(1+ 2^{j_{i}}t_{i}^{-1}\rho_{i}(x_{i},y_{i}))^{\lambda_{i} r}(1+ 2^{j_{i}}t_{i}^{-1} \rho_{i}(y_{i},z_{i}))^{\lambda_{i} r}\\
&\geq (1+ 2^{j_{i}}t_{i}^{-1}\rho_{i}(x_{i},z_{i}))^{\lambda_{i} r} \\
&\geq 2^{(j_{i}-k_{i})\lambda_{i}r}
(1+ 2^{k_{i}}t_{i}^{-1}\rho_{i}(x_{i},z_{i}))^{\lambda_{i} r} \quad (\forall k_{i} \geq j_{i})
\end{align*}
and $V(z_{i}, 2^{-j_{i}}t_{i}) \geq V(z_{i},2^{-k_{i}}t_{i})$ ($\forall k_{i} \geq j_{i}$), we obtain \eqref{cen} for $r >1$.
\end{proof}

\begin{lemma} \label{PEE}
Let $\Phi_{1},\Phi_{2} \in \mathcal{A}(\mathbb{R})$ be even functions.
Let $p \in (0, \infty)$ and $\lambda_{i} > \frac{(n_{i}+D_{i})q_{w}}{\min\{p,2\}}$, $i=1,2$.
Then there exits a constant $C$ such that for all $f \in L^{2}(X_{1} \times X_{2})$,
\begin{align*}
&\left\|\left(\int_{0}^{\infty} \int_{0}^{\infty}\left|\big[\Phi_{1}(t_{1}\sqrt{L_{1}})\otimes \Phi_{2}(t_{2}\sqrt{L_{2}})
\big]^{\ast}_{\lambda_{1},\lambda_{2}}f\right|^{2}\frac{dt_{1}}{t_{1}}\frac{dt_{2}}{t_{2}}
\right)^{1/2}\right\|_{L_{w}^{p}(X_{1} \times X_{2})}\\
&\quad \quad\quad\quad\quad\quad
\leq C \left\|\left(\int_{0}^{\infty} \int_{0}^{\infty}\left|\Phi_{1}(t_{1}\sqrt{L_{1}})\otimes \Phi_{2}(t_{2}\sqrt{L_{2}})
f\right|^{2}\frac{dt_{1}}{t_{1}}\frac{dt_{2}}{t_{2}}
\right)^{1/2}\right\|_{L_{w}^{p}(X_{1} \times X_{2})}.
\end{align*}
\end{lemma}
\begin{proof}
Since $\lambda_{i} > \frac{(n_{i}+D_{i})q_{w}}{\min\{p,2\}}$, there exists a number $r$ such
that $0< r < \frac{\min\{p,2\}}{q_{w}}$ and $\lambda_{i} r > n_{i}+D_{i}$.
From Lemma \ref{central} we see that for any $\sigma>0$
there exists a constant $C$ such that
for all $f \in L^{2}(X_{1} \times X_{2})$, $j_{i} \in \mathbb{Z}$, $x_{i}\in X_{i}$ and  $t_{i} \in [1,2]$,
\begin{align*}
&\left\{\big[\Phi_{1}(2^{-j_{1}}t_{1}\sqrt{L_{1}|})\otimes \Phi_{2}(2^{-j_{2}}t_{2}\sqrt{L_{2}})
\big]^{\ast}_{\lambda_{1},\lambda_{2}}f(x_{1}, x_{2})\right\}^{r} \\
&\leq C\sum_{k_{1}= j_{1}}^{\infty}\sum_{k_{2}=j_{2}}^{\infty} 2^{(j_{1}-k_{1})\sigma}2^{(j_{2}- k_{2})\sigma} \\
&\quad\quad \times \iint_{X_{1} \times X_{2}}  \frac{\left|\Phi_{1}(2^{-k_{1}} t_{1}\sqrt{L_{1}}) \otimes \Phi_{2}(2^{-k_{2}}
 t_{2}\sqrt{L_{2}}) f(z_{1},z_{2})\right|^{r}d\mu_{1}(z_{1})d\mu_{2}(z_{2})}
{V(z_{1},2^{- k_{1}}t_{1})(1+2^{k_{1}}t_{1}^{-1}\rho_{1}(x_{1},z_{1}))^{\lambda_{1}r}V(z_{2},2^{-k_{2}}t_{2})
(1+2^{k_{2}}t_{2}^{-1}\rho_{2}(x_{2},z_{2}))
^{\lambda_{2}r}}.
\end{align*}
Taking the norm $\big(\int_{1}^{2}\int_{1}^{2}|\cdot|^{2/r}\frac{dt_{1}}{t_{1}}\frac{dt_{2}}{t_{2}}\big)^{r/2}$ on
both sides, applying Minkowski's inequality, and then using \eqref{smfx},  we get
\begin{align*}
& \left(\int_{1}^{2}\int_{1}^{2}\left|\big[\Phi_{1}(2^{-j_{1}}t_{1}\sqrt{L_{1}})\otimes \Phi_{2}(2^{-j_{2}}t_{2}\sqrt{L_{2}})
\big]^{\ast}_{\lambda_{1},\lambda_{2}}f(x_{1},x_{2})\right|^{2}\frac{dt_{1}}{t_{1}}\frac{dt_{2}}{t_{2}}\right)^{r/2} \\
&\leq  C\sum_{k_{1}= j_{1}}^{\infty}\sum_{k_{2}= j_{2}}^{\infty} 2^{(j_{1}-k_{1})\sigma}2^{(j_{2}-k_{2})\sigma}\\
&\quad\quad \times \iint_{X_{1} \times X_{2}}  \frac{\left(\int_{1}^{2}\int_{1}^{2}\left|\Phi_{1}(2^{-k_{1}} t_{1}\sqrt{L_{1}})
\otimes \Phi_{2}(2^{-k_{2}} t_{2}\sqrt{L_{2}})f(z_{1},z_{2})\right|^{2}\frac{dt_{1}}{t_{1}}\frac{dt_{2}}{t_{2}}\right)^{r/2}d\mu_{1}(z_{1})d\mu_{2}(z_{2})}
{V(z_{1},2^{-k_{1}}t_{1})(1+2^{k_{1}}t_{1}^{-1}\rho_{1}(x_{1},z_{1}))^{\lambda_{1}r}V(z_{2},2^{-k_{2}}t_{2})(1+2^{k_{2}}t_{2}^{-1}\rho_{2}(x_{2},z_{2}))
^{\lambda_{2}r}} \\
&\leq  C\sum_{k_{1}= j_{1}}^{\infty}\sum_{k_{2}=j_{2}}^{\infty}
 2^{(j_{1}-k_{1})\sigma}2^{(j_{2}-k_{2})\sigma} \\
&\quad\quad \times \mathcal{M}_{s}\left[\left(\int_{1}^{2}\int_{1}^{2}\left|\Phi_{1}(2^{-k_{1}} t_{1}
\sqrt{L_{1}}) \otimes \Phi_{2}(2^{-k_{2}} t_{2}\sqrt{L_{2}}) f\right|^{2}\frac{dt_{1}}{t_{1}}\frac{dt_{2}}{t_{2}}\right)^{r/2}\right](x_{1},x_{2})\\
&\leq  C\sum_{k_{1}=-\infty}^{\infty}\sum_{k_{2}=-\infty}^{\infty} 2^{-|k_{1}-j_{1}|\sigma}2^{-|k_{2}-j_{2}|\sigma} \\
&\quad\quad \times
\mathcal{M}_{s}\left[\left(\int_{1}^{2}\int_{1}^{2}\left|\Phi_{1}(2^{-k_{1}} t_{1}
\sqrt{L_{1}}) \otimes \Phi_{2}(2^{-k_{2}} t_{2}\sqrt{L_{2}}) f\right|^{2}
\frac{dt_{1}}{t_{1}}\frac{dt_{2}}{t_{2}}\right)^{r/2}\right](x_{1},x_{2}).
\end{align*}
It then follows from Lemma \ref{RY} and Lemma \ref{maximal} that
\begin{align*}
&\left\|\left(\int_{0}^{\infty}\int_{0}^{\infty}\left|\big[\Phi_{1}(t_{1}\sqrt{L_{1}})\otimes \Phi_{2}(t_{2}\sqrt{L_{2}})
\big]^{\ast}_{\lambda_{1},\lambda_{2}}f\right|^{2}\frac{dt_{1}}{t_{1}}\frac{dt_{2}}{t_{2}}
\right)^{1/2}\right\|_{L_{w}^{p}(X_{1} \times X_{2})} \\
&  =\left\| \left\{\left(\int_{1}^{2}\int_{1}^{2}\left|\big[\Phi_{1}(2^{-j_{1}}t_{1}\sqrt{L_{1}})\otimes \Phi_{2}(2^{-j_{2}}
t_{2}\sqrt{L_{2}})\big]^{\ast}_{\lambda_{1},\lambda_{2}}f\right|^{2}\frac{dt_{1}}{t_{1}}\frac{dt_{2}}{t_{2}}\right)^{r/2}
\right\}_{j_{1},j_{2} \in \mathbb{Z}} \right\|_{L_{w}^{p/r}(\ell^{2/r})}^{1/r} \\
& \leq C \left\| \left\{\mathcal{M}_{s}\left[\left(\int_{1}^{2}\int_{1}^{2}\left|\Phi_{1}(2^{-j_{1}} t_{1}\sqrt{L_{1}})
\otimes \Phi_{2}(2^{-j_{2}} t_{2}\sqrt{L_{2}})f\right|^{2}\frac{dt_{1}}{t_{1}}\frac{dt_{2}}{t_{2}}
\right)^{r/2}\right]\right\}_{j_{1}, j_{2} \in \mathbb{Z}} \right\|_{L_{w}^{p/r}(\ell^{2/r})}^{1/r}\\
& \leq C \left\| \left\{\left(\int_{1}^{2}\int_{1}^{2}\left|\Phi_{1}(2^{-j_{1}} t_{1}\sqrt{L_{1}}) \otimes
\Phi_{2}(2^{-j_{2}} t_{2}\sqrt{L_{2}}) f\right|^{2}\frac{dt_{1}}{t_{1}}\frac{dt_{2}}{t_{2}}\right)^{r/2}\right\}_{j_{1},j_{2} \in \mathbb{Z}} \right\|_{L_{w}^{p/r}(\ell^{2/r})}^{1/r}\\
& = C \left\| \left(\int_{0}^{\infty}\int_{0}^{\infty}\left|\Phi_{1}( t_{1}\sqrt{L_{1}}) \otimes \Phi_{2}(t_{2}\sqrt{L_{2}}) f\right|^{2}\frac{dt_{1}}{t_{1}}\frac{dt_{2}}{t_{2}}\right)^{1/2}\right\|_{L_{w}^{p}(X_{1}\times X_{2})},
\end{align*}
where we used the fact that $p/r>q_{w}$ (which implies $w \in A_{p/r}(X_{1} \times X_{2})$) and $2/r>1$.
\end{proof}

\begin{lemma} \label{lkjh}
Let $\Phi_{1},\Phi_{2} \in \mathcal{A}(\mathbb{R})$ be even functions.
Let $p \in (0, \infty)$ and $\lambda_{i}>0$, $i=1,2$. Let $w$ be arbitrary weight (i.e., non-negative locally integrable function)
on $X_{1} \times X_{2}$. Then there exists a constant $C$ such that for all $f \in L^{2}(X_{1} \times X_{2})$,
\begin{align*}
&\left\|\left(\int_{0}^{\infty} \int_{0}^{\infty}\left|\big[\Phi_{1}(t_{1}^{2}\sqrt{L_{1}})\otimes \Phi_{2}(t_{2}\sqrt{L_{2}})
\big]^{\ast}_{\lambda_{1}+D_{1}/2,\lambda_{2}+D_{2}/2}f\right|^{2}\frac{dt_{1}}{t_{1}}\frac{dt_{2}}{t_{2}}
\right)^{1/2}\right\|_{L_{w}^{p}(X_{1} \times X_{2})} \\
&\leq C\left\|g^{\ast}_{\Phi_{1},\Phi_{2}, L_{1},L_{2},(2/n_{1})\lambda_{1}, (2/n_{2})\lambda_{2}}(f)\right\|_{L_{w}^{p}(X_{1}\times X_{2})}.
\end{align*}
\end{lemma}

\begin{proof}
Let $\sigma> 0$.
By Lemma \ref{central} with $r =2$, we see that
there exists a constant $C$ such that for all $f \in L^{2}(X_{1} \times X_{2})$,
$j_{i} \in \mathbb{Z}$ and $t_{i} \in [1,2]$,
\begin{align} \label{098}
 &\left\{\big[\Phi_{1}(2^{- j_{1}}t_{1}\sqrt{L_{1}})\otimes \Phi_{2}(2^{-j_{2}}t_{2}\sqrt{L_{2}})
\big]^{\ast}_{\lambda_{1}+D_{1}/2,\lambda_{2}+D_{2}/2}f(x_{1}, x_{2}) \right\}^{2}\\
&\leq C \sum_{k_{1}=j_{1} }^{\infty}\sum_{k_{2}= j_{2}}^{\infty}2^{(j_{1}-k_{1})\sigma}2^{(j_{2}-k_{2})\sigma}  \nonumber\\
& \quad\quad \times \iint_{X_{1}\times X_{2}}
\frac{\left|\Phi_{1}(2^{-k_{1}}t_{1}\sqrt{L_{1}})\otimes \Phi_{2}(2^{-k_{2}}t_{2}\sqrt{L_{2}})
f(z_{1},z_{2})\right|^{2}}{\prod_{i=1}^{2}V(z_{i},2^{-k_{i}}t_{i})(1 + 2^{k_{i}}t_{i}^{-1}\rho_{i}(x_{i},z_{i}))^{2\lambda_{i} +D_{i}}
}d\mu_{1}(z_{1})d\mu_{2}(z_{2})\nonumber\\
&\leq C \sum_{k_{1} =-\infty}^{\infty} \sum_{k_{2} =-\infty}^{\infty}2^{-|k_{1}- j_{1}|\sigma}2^{-|k_{2}-j_{2}|\sigma}\nonumber \\
& \quad\quad \times \iint_{X_{1}\times X_{2}}
\frac{\left|\Phi_{1}(2^{-k_{1}}t_{1}\sqrt{L_{1}})\otimes \Phi_{2}(2^{-k_{2}}t_{2}
\sqrt{L_{2}})f(z_{1},z_{2})\right|^{2}}{\prod_{i=1}^{2}(1 + 2^{k_{i}}t_{i}^{-1}\rho(x_{i},z_{i}))^{2\lambda_{i}}}
\frac{d\mu_{1}(z_{1})}{V(x_{1},2^{-k_{1}}t_{1})}
\frac{d\mu_{2}(z_{2})}{V(x_{2},2^{-k_{2}}t_{2})}, \nonumber
\end{align}
where for the last line we used \eqref{tos}.
Taking the norm $\int_{1}^{2}\int_{1}^{2} |\cdot| \frac{dt_{1}}{t_{1}}\frac{dt_{2}}{t_{2}}$ on both sides of \eqref{098} gives
\begin{align*}
&\int_{1}^{2}\int_{1}^{2} \left\{\big[\Phi_{1}(2^{-2j_{1}}t_{1}\sqrt{L_{1}})\otimes \Phi_{2}(2^{-j_{2}}t_{2}\sqrt{L_{2}})
\big]^{\ast}_{\lambda_{1}+D_{1}/2,\lambda_{2}+D_{2}/2}f(x_{1}, x_{2}) \right\}^{2}\frac{dt_{1}}{t_{1}}\frac{dt_{2}}{t_{2}} \\
&\leq C \sum_{k_{1} =-\infty}^{\infty}\sum_{k_{2} =-\infty}^{\infty}2^{-|k_{1}- j_{1}|\sigma}2^{-|k_{2}- j_{2}|\sigma} \\
& \quad \times \int_{1}^{2}\!\int_{1}^{2}\!\int_{X_{1}}\!\int_{X_{2}}
\frac{\left|\Phi_{1}(2^{-k_{1}}t_{1}\sqrt{L_{1}})\otimes \Phi_{2}(2^{-k_{2}}t_{2}\sqrt{L_{2}})
f(z_{1},z_{2})\right|^{2}}{(1 + 2^{k_{1}}t_{1}^{-1}\rho(x_{1},z_{1}))^{2\lambda_{1}}
(1 + 2^{k_{2}}t_{2}^{-1}\rho_{2}(x_{2},z_{2}))^{2\lambda_{2}}}\frac{d\mu_{1}(z_{1})dt_{1}}{V(x_{1},2^{-k_{1}}t_{1})t_{1}}
\frac{d\mu_{2}(z_{2})dt_{2}}{V(x_{2},2^{-k_{2}}t_{2})t_{2}}.
\end{align*}
Applying Lemma \ref{RY} in $L_{w}^{p/2}(\ell^{1})$ we obtain
\begin{align*}
&\left\|\left(\int_{0}^{\infty} \int_{0}^{\infty}\left|\big[\Phi_{1}(t_{1}\sqrt{L_{1}})\otimes \Phi_{2}(t_{2}\sqrt{L_{2}})
\big]^{\ast}_{\lambda_{1}+D_{1}/2,\lambda_{2}+D_{2}/2}f\right|^{2}\frac{dt_{1}}{t_{1}}\frac{dt_{2}}{t_{2}}
\right)^{1/2}\right\|_{L_{w}^{p}(X_{1} \times X_{2})} \\
&=\left\|\left\{\int_{1}^{2}\int_{1}^{2} \left\{\big[\Phi_{1}(2^{-j_{1}}t_{1}\sqrt{L_{1}})\otimes
\Phi_{2}(2^{-j_{2}}t_{2}\sqrt{L_{2}})
\big]^{\ast}_{\lambda_{1}+D_{1}/2,\lambda_{2}+D_{2}/2}f(x_{1}, x_{2}) \right\}^{2}\frac{dt_{1}}{t_{1}}\frac{dt_{2}}
{t_{2}}\right\}_{j_{1},j_{2}\in \mathbb{Z}}
\right\|_{L_{w}^{p/2}(\ell^{1})}^{1/2}\\
&\leq C \left\|\left\{\int_{1}^{2}\int_{1}^{2}\int_{X}\int_{X}
\frac{\left|\Phi_{1}(2^{-k_{1}}t_{1}\sqrt{L_{1}})\otimes \Phi_{2}(2^{-k_{2}}t_{2}\sqrt{L_{2}})
f(z_{1},z_{2})\right|^{2}}{(1 + 2^{k_{1}}t_{1}^{-1}\rho(x_{1},z_{1}))^{2\lambda_{1}}
(1 + 2^{k_{2}}t_{2}^{-2}\rho_{2}(x_{2},z_{2}))^{2\lambda_{2}}} \right.\right.\\
& \quad\quad\quad\quad\quad\quad\quad\quad\quad\quad\quad\quad\quad\quad\quad
\times \left.\left.\frac{d\mu_{1}(z_{1})dt_{1}}{V(x_{1},2^{-k_{1}}t_{1})t_{1}}
 \frac{d\mu_{2}(z_{2})dt_{2}}{V(x_{2},2^{-k_{2}}t_{2})t_{2}}\right\}
_{k_{1},k_{2}\in \mathbb{Z}}\right\|_{L_{w}^{p/2}(\ell^{1})}^{1/2} \\
& = C\left\|g^{\ast}_{\Phi_{1},\Phi_{2},L_{1},L_{2}, (2/n_{1})\lambda_{1}, (2/n_{2})\lambda_{2}}(f)\right\|_{L_{w}^{p}(X_{1}\times X_{2})},
\end{align*}
as desired.
\end{proof}

Having the above lemmas, we are ready to give the proofs of Theorems \ref{main1} and \ref{main2}.

\begin{proof}[Proof of Theorem \ref{main1}]
Let $\Phi_{1}, \Phi_{2}, \widetilde{\Phi}_{1}, \widetilde{\Phi}_{2} \in \mathcal{A}(\mathbb{R})$ be even functions satisfying
\begin{equation*}
\Phi_{1}(0) = \Phi_{2}(0) =\widetilde{\Phi}_{1}(0) = \widetilde{\Phi}_{2}(0)= 0.
\end{equation*}
Let $p \in (0,\infty)$ and $\lambda_{i} > \frac{(n_{i}+D_{i})q_{w}}{\min\{p,2\}}$, $i=1,2$.
Note that for a.e. $(x_{1}, x_{2}) \in X_{1} \times X_{2}$,
\begin{equation} \label{haji}
\widetilde{\Phi}_{1}(t_{1}\sqrt{L_{1}})\otimes \widetilde{\Phi}_{2}(t_{2}\sqrt{L_{2}})f(x_{1},x_{2}) \leq
\big[\widetilde{\Phi}_{1}(t_{1}\sqrt{L_{1}})\otimes \widetilde{\Phi}_{2}(t_{2}\sqrt{L_{2}})
\big]^{\ast}_{\lambda_{1},\lambda_{2}}f(x_{1},x_{2}).
\end{equation}
Using \eqref{haji}, Lemma \ref{211} and Lemma \ref{PEE}, we infer
\begin{align*}
&\|g_{\widetilde{\Phi}_{1},\widetilde{\Phi}_{2},L_{1},L_{2}}(f)\|_{L_{w}^{p}(X_{1}\times X_{2})} \\
& =\left\|\left(\int_{0}^{\infty} \int_{0}^{\infty}\left|\widetilde{\Phi}_{1}(t_{1}\sqrt{L_{1}})\otimes \widetilde{\Phi}_{2}(t_{2}\sqrt{L_{2}})
f\right|^{2}\frac{dt_{1}}{t_{1}}\frac{dt_{2}}{t_{2}}
\right)^{1/2}\right\|_{L_{w}^{p}(X_{1} \times X_{2})} \\
&\leq \left\|\left(\int_{0}^{\infty} \int_{0}^{\infty}\left|\big[\widetilde{\Phi}_{1}(t_{1}\sqrt{L_{1}})\otimes \widetilde{\Phi}_{2}(t_{2}\sqrt{L_{2}})
\big]^{\ast}_{\lambda_{1},\lambda_{2}}f\right|^{2}\frac{dt_{1}}{t_{1}}\frac{dt_{2}}{t_{2}}
\right)^{1/2}\right\|_{L_{w}^{p}(X_{1} \times X_{2})} \\
& \leq C\left\|\left(\int_{0}^{\infty} \int_{0}^{\infty}\left|\big[\Phi_{1}(t_{1}\sqrt{L_{1}})\otimes \Phi_{2}(t_{2}\sqrt{L_{2}})
\big]^{\ast}_{\lambda_{1},\lambda_{2}}f\right|^{2}\frac{dt_{1}}{t_{1}}\frac{dt_{2}}{t_{2}}
\right)^{1/2}\right\|_{L_{w}^{p}(X_{1} \times X_{2})} \\
&\leq C\left\|\left(\int_{0}^{\infty} \int_{0}^{\infty}\left|\Phi_{1}(t_{1}\sqrt{L_{1}})\otimes \Phi_{2}(t_{2}\sqrt{L_{2}})
f\right|^{2}\frac{dt_{1}}{t_{1}}\frac{dt_{2}}{t_{2}}
\right)^{1/2}\right\|_{L_{w}^{p}(X_{1} \times X_{2})} \\
& = C\|g_{\Phi_{1},\Phi_{2},L_{1},L_{2}}(f)\|_{L_{w}^{p}(X_{1}\times X_{2})}.
\end{align*}
By symmetry, there also holds $\|g_{\Phi_{1},\Phi_{2},L_{1},L_{2}}(f)\|_{L_{w}^{p}(X_{1}\times X_{2}) }
\leq C \|g_{\Psi_{1},\Psi_{2},L_{1},L_{2}}(f)\|_{L_w{}^{p}(X_{1}\times X_{2})}$. Hence the assertion of Theorem \ref{main1} is true.
\end{proof}

\begin{proof}[Proof of Theorem \ref{main2}]
Let $\Phi_{1}, \Phi_{2} \in \mathcal{A}(\mathbb{R})$ be even functions.
Let $p \in (0, \infty)$, $\lambda_{i} >\frac{2q_{w}}{\min\{p,2\}}$
and $\lambda_{i}' > \frac{(n_{i}+D_{i})q_{w}}{\min\{p,2\}}$, $i=1,2$.
Then, for all $f\in L^{2}(X_{1} \times X_{2})$,
by \eqref{haji}, Lemma \ref{lkjh}, Lemma \ref{lkjhh}, Lemma \ref{8787} and Lemma \ref{PEE}, we have
\begin{align*}
&\|g_{\Phi_{1},\Phi_{2},L_{1},L_{2}}(f)\|_{L_{w}^{p}(X_{1}\times X_{2})} \\
& \leq \left\|\left(\int_{0}^{\infty} \int_{0}^{\infty}\left|\big[\Phi_{1}(t_{1}
\sqrt{L_{1}})\otimes \Phi_{2}(t_{2}\sqrt{L_{2}})\big]^{\ast}_{(n_{1}/2)\lambda_{1}+D_{1}/2,(n_{2}/2)\lambda_{2}+D_{2}/2}f\right|^{2}
\frac{dt_{1}}{t_{1}}\frac{dt_{2}}{t_{2}} \right)^{1/2}\right\|_{L_{w}^{p}(X_{1} \times X_{2})} \\
& \leq C\left\|g^{\ast}_{\Phi_{1},\Phi_{2}, L_{1},L_{2},\lambda_{1}, \lambda_{2}}(f)\right\|_{L_{w}^{p}(X_{1}\times X_{2})}\\
&\leq C\big\|S_{\Phi_{1},\Phi_{2},L_{1},L_{2}}(f)\big\|_{L_{w}^{p}(X_{1}\times X_{2})}\\
& \leq C\left\|\left(\int_{0}^{\infty} \int_{0}^{\infty}\left|\big[\Phi_{1}(t_{1}\sqrt{L_{1}})\otimes \Phi_{2}(t_{2}
\sqrt{L_{2}})\big]^{\ast}_{\lambda_{1}',\lambda_{2}'}f\right|^{2}\frac{dt_{1}}{t_{1}}\frac{dt_{2}}{t_{2}}
\right)^{1/2}\right\|_{L_{w}^{p}(X_{1} \times X_{2})} \\
& \leq C\|g_{\Phi_{1},\Phi_{2},L_{1},L_{2}}(f)\|_{L_{w}^{p}(X_{1}\times X_{2})},
\end{align*}
which yields \eqref{man}. The proof of Theorem \ref{main2} is complete.
\end{proof}

\section{Applications of Theorems \ref{main1} and \ref{main2}}

1. In  \cite{DSTY} and \cite{DLY}, the theory of product Hardy space $H^1_{L_1,L_2}(\mathbb R^n\times \mathbb R^m)$ via the Littlewood--Paley area functions
were established, where $L_1$ and $L_2$ are two non-negative self-adjoint operators that satisfy only the Gaussian heat kernel bound. To be more specific,
$H^1_{L_1,L_2}(\mathbb R^n\times \mathbb R^m)$ is defined as the closure of
$$ \{ f\in L^2(\mathbb R^n\times \mathbb R^m): S_{L_1,L_2}(f)\in L^1(\mathbb R^n\times \mathbb R^m) \} $$
under the norm
$ \|f\|_{H^1_{L_1,L_2}(\mathbb R^n\times \mathbb R^m)}:=  \|S_{L_1,L_2}(f)\|_{H^1_{L_1,L_2}(\mathbb R^n\times \mathbb R^m)}$, where
$$S_{L_1,L_2}(f)(x_1,x_2)=  \left(\iint_{\Gamma_{1}(x_{1}) \times \Gamma_{2}(x_{2})}\left| ( t_{1}^{2}L_{1}e^{-t_{1}^{2}L_{1}}) \otimes
(t_{2}^{2}L_{2}e^{-t_{2}^{2}L_{2}})f(y_{1},y_{2})  \right|^{2}  \frac{dy_{1}dt_{1}}{t_{1}^{n+1}} \frac{dy_{2}dt_{2}}{t_{2}^{m+1}} \right)^{1/2}.$$
Then, by applying our main result Theorem \ref{main2} (also Remark \ref{remark}), we obtain the characterization of $H^1_{L_1,L_2}(\mathbb R^n\times \mathbb R^m)$ via the Littlewood--Paley square function as follows, which is missing in \cite{DSTY} and \cite{DLY}, i.e., $H^1_{L_1,L_2}(\mathbb R^n\times \mathbb R^m)$ is equivalent to the closure of
$$ \left\{ f\in L^2(\mathbb R^n\times \mathbb R^m): g_{L_1,L_2}(f)\in L^1(\mathbb R^n\times \mathbb R^m) \right\} $$
under the norm
$  \|g_{L_1,L_2}(f)\|_{H^1_{L_1,L_2}(\mathbb R^n\times \mathbb R^m)}$, where
$$g_{L_1,L_2}(f)(x_1,x_2)=  \left(\int_0^\infty\int_0^\infty \left| ( t_{1}^{2}L_{1}e^{-t_{1}^{2}L_{1}}) \otimes
(t_{2}^{2}L_{2}e^{-t_{2}^{2}L_{2}})f(x_{1},x_{2})  \right|^{2}  \frac{dt_{1}}{t_{1}} \frac{dt_{2}}{t_{2}} \right)^{1/2}.$$

2. In 1965, Muckenhoupt and Stein in \cite{ms} introduced a notion of conjugacy associated with the Bessel operator $\tbz$ on $\mathbb R_+:=(0,\infty)$
defined by
\begin{equation*}
\tbz f(x):=-\frac{d^2}{dx^2}f(x)-\frac{2\lz}{x}\frac{d}{dx}f(x),\quad x>0,
\end{equation*}
and the  Bessel Schr\"odinger operator $S_\lz$ on $\mathbb R_+$
\begin{equation*}
S_\lambda f(x):=-\frac{d^2}{dx^2}f(x)+ \frac{\lz^2-\lz}{x^2}f(x),\quad x>0.
\end{equation*}
In \cite{DLWY}, Duong {\it et al.} established the product Hardy space $H^p_{\tbz}(\mathbb R_+\times \mathbb R_+)$ associated with $\tbz$ via the
Littlewood--Paley area function and square functions. Note that the measure on $\mathbb R_+$ related to $\tbz$ is $d\mu_\lz(x)= x^{2\lz}dx$.
We point out that
the kernel of $t^{2}\tbz e^{-t^{2}\tbz}$ satisfies the Gaussian upper bounds with respect to the measure $d\mu_\lz$, the H\"older regularity and the cancellation property. Hence, by using the approach in \cite{HLW} via the Plancherel--Polya type inequality, they obtained the equivalence of the
 characterizations of $H^p_{\tbz}(\mathbb R_+\times \mathbb R_+)$  via Littlewood--Paley area function function and square functions.
 By applying our main result Theorem \ref{main2} (also Remark \ref{remark}), we obtain a direct proof of the equivalence without using the H\"older regularity and the cancellation property.

In \cite{BDLWY}, Betancor {\it et al.} established the product  Hardy space $H^p_{S_\lz}(\mathbb R_+\times \mathbb R_+)$ associated with $\tbz$ via the
Littlewood--Paley area function and square functions. To prove the equivalence, they need to use the Poisson semigroup $\{e^{-t\sqrt{S_\lz}}\}$, the subordination formula and the Moser type inequality as a bridge. By applying our main result Theorem \ref{main2} (also Remark \ref{remark}), we obtain a direct proof of this equivalence without using the Moser type inequality.

\bigskip
{\bf Acknowledgements}: X.T. Duong and J. Li are supported by DP 160100153. G. Hu is supported by Tianyuan Fund for Mathematics
of China, No. 11626122.
\bibliographystyle{amsplain}

\end{document}